\documentclass[aop]{imsart}

\RequirePackage{amsthm,amsmath,amsfonts,amssymb,textgreek}
\RequirePackage[numbers]{natbib}

\RequirePackage[colorlinks,citecolor=blue,urlcolor=blue]{hyperref}
\RequirePackage{graphicx}

\startlocaldefs

\numberwithin{equation}{section}

\theoremstyle{plain}

\newtheorem{theorem}{Theorem}[section]
\newtheorem{lemma}[theorem]{Lemma}
\newtheorem{proposition}[theorem]{Proposition}
\newtheorem{corollary}[theorem]{Corollary}

\theoremstyle{remark}
\newtheorem{remark}{Remark}
\newtheorem{definition}[theorem]{Definition}

  \renewcommand{\leq}{\leqslant}
\renewcommand{\geq}{\geqslant}

\newcommand{\f}[1]{\mathfrak{#1}}
\newcommand{\ip}[2]{\left\langle #1, #2\right\rangle }
\newcommand{\no}[1]{\left\Vert #1\right\Vert  }
\newcommand{\ed}{\overset{{\scriptscriptstyle \mathrm{d}}}{=}}

\endlocaldefs

\begin{document}

\begin{frontmatter}
\title{Symmetric Stable Processes on Amenable Groups}
\runtitle{S$\alpha$S Processes on Amenable Groups}

\begin{aug}

\author{Nachi Avraham-Re'em}

\address{Einstein Institute of Mathematics,
The Hebrew University of Jerusalem,\\ \href{mailto:nachi.avraham@gmail.com}{nachi.avraham@gmail.com}}

\end{aug}

\begin{abstract}
We show that if $G$ is a countable amenable group, then every stationary non-Gaussian symmetric $\alpha$-stable ($S\alpha S$) process indexed by $G$ is ergodic if and only if it is weakly-mixing, and it is ergodic if and only if its Rosiński minimal spectral representation is null. This extends previous results for $\mathbb{Z}^d$, and answers a question of P. Roy on discrete nilpotent groups to the extent of all countable amenable groups. As a result, we construct on the Heisenberg group and on many Abelian groups, for all $\alpha\in\left(0,2\right)$, stationary $S\alpha S$ processes that are weakly-mixing but not strongly-mixing.
\end{abstract}



\end{frontmatter}

\setcounter{tocdepth}{1}
\tableofcontents

\section{Introduction and main results}

Let $G$ be a countable group. A real-valued stochastic process $\mathbf{X}=\left(X_{g}\right)_{g\in G}$ indexed by $G$ (henceforth, a {\it $G$-process}) is called {\it symmetric $\alpha$-stable} (henceforth, $S\alpha S$) for some $\alpha\in\left(0,2\right]$, if the distribution of every linear combination of the coordinates of $\mathbf{X}$ has an $S\alpha S$ distribution. Recall that a random variable $X$ has the $S\alpha S$ distribution for $\alpha\in\left(0,2\right]$ if, for some $\sigma>0$, 
$$\mathbb{E}\left(\exp\left(i\theta X\right)\right)=\exp\left(-\sigma^{\alpha}\left|\theta\right|^{\alpha}\right),\quad \theta\in\mathbb{R}.$$
A $G$-process is (left) {\it stationary} if $\left(X_g\right)_{g\in G}\ed\left(X_{h^{-1} g}\right)_{g\in G}$ for every $h\in G$, where $\ed$ means equality in distribution.

It was shown by Podgórski \& Weron \cite{podgorski, podwer} that, for stationary $S\alpha S$ processes indexed by $\mathbb{Z}$ and $\mathbb{R}$, ergodicity and weak-mixing are equivalent. See Definitions \ref{Definition: Ergodicity} and \ref{Definition: Weak-mixing}. This was generalized to $\mathbb{Z}^d$ and $\mathbb{R}^d$ by Wang, Roy \& Stoev \cite{wangroystoev}. For non-Abelian groups it was unknown and this brings us to the first result.

\begin{theorem}
\label{Theorem: Ergodicity iff Weak-Mixing}
Let $G$ be a countable amenable group. A stationary real-valued $S\alpha S$ $G$-process is ergodic if and only if it is weakly-mixing.
\end{theorem}

In his seminal work \cite{rosinski1995}, Rosiński, following Hardin \cite{hardin}, established a bridge between the theory of stationary non-Gaussian $S\alpha S$ $G$-processes and non-singular ergodic theory. By {\it non-Gaussian} we refer to the case $\alpha\neq 2$. The interplay between the theories has been found to be rich and fruitful, and it opened the door to the study of probabilistic properties of stationary $S\alpha S$ processes using non-singular ergodic theory. See \cite{rossam, rosinski2000, roy2007, roy2012, bhattroy, roy2020} and the references in \cite{roy2020}. A notable instance of such an interplay is Samorodnitsky's characterization of ergodicity and weak-mixing in terms of the {\it Rosiński minimal spectral representation} of the process, which we introduce now.

It was shown in \cite{kuelbs, schilder} that for every $S\alpha S$ $G$-process $\mathbf{X}$ there exists a standard σ-finite measure space $\left(\Omega,\mu\right)$ and a collection $\left(f_{g}\right)_{g\in G}\subset L^{\alpha}\left(\mu\right)$ such that
$$\mathbf{X}\ed\left(\int_{\Omega}f_{g}dM_{\alpha}\right)_{g\in G}.$$
Here $M_{\alpha}$ is an $S\alpha S$ random measure controlled by $\mu$. Such a representation will be referred to as a {\it spectral representation} of $\mathbf{X}$. See \cite[\textsection 3]{samtakku} for details on $S\alpha S$ random measures and stochastic integrals. Following Hardin \cite{hardin}, such a spectral representation is typically chosen to be {\it minimal} as in Definition \ref{Definition: Spectral Representation}.

Recall that a ({\it non-singular}) {\it $G$-system} is a group action $G\overset{\f a}{\curvearrowright}\left(\Omega,\mu\right)$ of $G$ on a standard σ-finite measure space $\left(\Omega,\mu\right)$ by {\it non-singular automorphisms}. Namely, each mapping $\f a_{g}:\Omega\to\Omega$ is a bi-measurable bijection which is non-singular in the sense that the measures $\mu\circ\f a_{g}$ and $\mu$ are mutually absolutely continuous. We will consider $G$-systems $G\overset{\f a^{\left(1\right)}}{\curvearrowright}\left(\Omega_{1},\mu_{1}\right)$ and $G\overset{\f a^{\left(2\right)}}{\curvearrowright}\left(\Omega_{2},\mu_{2}\right)$ as isomorphic within the category of (non-singular) $G$-systems, if there exists a bi-measurable bijection $F:\Omega_{1}\to\Omega_{2}$, defined up to zero measure sets, such that (1) $F\circ\f a_{g}^{\left(1\right)}=\f a_{g}^{\left(2\right)}\circ F$ almost surely for every $g\in G$ and (2) the measures $\mu_{2}$ and $\mu_{1}\circ F^{-1}$ on $\Omega_{2}$ are mutually absolutely continuous.

Rosiński showed that the minimal spectral representation of every real-valued stationary non-Gaussian $S\alpha S$ $G$-process $\mathbf{X}$ is of the form
\begin{equation}
    \label{eq: Rosinski Representation}
    \mathbf{X}\ed \left(\intop_{\Omega}c_{g}\sqrt[\alpha]{\frac{d\mu\circ\f a_{g}}{d\mu}}f_0\circ\f a_{g} dM_{\alpha}\right)_{g\in G},
\end{equation}
for some $G$-system $G\overset{\f a}{\curvearrowright}\left(\Omega,\mu\right)$, a cocycle $c:G\times\Omega\to\left\{+1,-1\right\}$ and a special function $f_0\in L^{\alpha}\left(\mu\right)$. Furthermore, this representation is unique up to an isomorphism of $G$-systems. See Theorem \ref{Theorem: Rosiński} below. We will refer to the (minimal) spectral representation as in \eqref{eq: Rosinski Representation} by {\it Rosiński (minimal) spectral representation}.

Given a $G$-system $G\overset{\f a}{\curvearrowright}\left(\Omega,\mu\right)$ and a fixed function $f_0\in L^{\alpha}\left(\mu\right)$, one can define a stationary $S\alpha S$ $G$-process via
\begin{equation}
    \label{eq: represeted process}
    \mathbf{X}:=\left(\intop_{\Omega}\sqrt[\alpha]{\frac{d\mu\circ\f a_{g}}{d\mu}}f_0\circ\f a_{g} dM_{\alpha}\right)_{g\in G},
\end{equation}
where $M_{\alpha}$ is an $S\alpha S$ random measure controlled by $\mu$. However, while here we are given a spectral representation of $\mathbf{X}$, checking its minimality is difficult (cf. \cite{rosinski1995,rosinski2006}) and for that purpose the following theorem is convenient. A $G$-system $G\overset{\f a}{\curvearrowright}\left(\Omega,\mu\right)$ is said to be {\it null} if it has no absolutely continuous invariant probability measure (a.c.i.p.), that is a Borel probability measure $\nu$ on $\Omega$ which is both: absolutely continuous with respect to $\mu$ and $\f a$-invariant.

\begin{theorem}
\label{Theorem: weak-mixing null}
Let $G$ be a countable amenable group and $G\overset{\f a}{\curvearrowright}\left(\Omega,\mu\right)$ a null $G$-system. Then every stationary $S\alpha S$ $G$-process given by \eqref{eq: represeted process} is weakly-mixing.
\end{theorem}

In order to formulate the characterization of ergodicity and weak-mixing in terms of the Rosiński minimal spectral representation in the general case, Samorodnitsky used the positive-null decomposition (see Section \ref{Section: Positive-Null Decomposition}), showing that for stationary non-Gaussian $S\alpha S$ processes indexed by $\mathbb{Z}$ or $\mathbb{R}$ ergodicity is equivalent to the process having Rosiński minimal spectral representation that is null. That is to say, the non-singular $G$-system underlying the Rosiński minimal spectral representation is null \cite{samorodnitsky2005}. This result was also generalized to $\mathbb{Z}^d$ and $\mathbb{R}^d$ by Wang, Roy \& Stoev \cite{wangroystoev}, using a different approach to the positive-null decomposition following Takahashi \cite{takahashi}.

In a recent work \cite{roy2020}, Roy studied {\it non-commutative} aspects of stationary non-Gaussian $S\alpha S$ $G$-processes, establishing a connection between this theory and operator algebras via the group measure space construction of Murray-von Neumann. It was conjectured by Roy that the above characterization holds true for the discrete Heisenberg group and he also posed the problem of generalizing this to a larger class of amenable groups such as nilpotent groups. The following theorem confirms this to the extent of all countable amenable groups.

\begin{theorem}
\label{Theorem: Main Theorem}
Let $G$ be a countable amenable group. A stationary real-valued non-Gaussian $S\alpha S$ $G$-process is weak-mixing if and only if its Rosiński minimal spectral representation is null.
\end{theorem}

Note that in the formulation of Theorem \ref{Theorem: Main Theorem} we used the ergodic-theoretic fact that a positive-null decomposition indeed exists for $G$-systems with $G$ a countable amenable group. This will be proved in Section \ref{Section: Positive-Null Decomposition}.

\begin{remark}
After this work was first published, I was informed by M. Mj, P. Roy and S. Sarkar that they have also been working on a proof of Theorem \ref{Theorem: Main Theorem}. Their results can be found in \cite{mj2022}.
\end{remark}

\subsection*{About the proofs}

The large collection of works of the aforementioned authors on the theory of $S\alpha S$ processes provides us many probabilistic tools that we may apply also in the amenable world. Thus, our main challenges here are functional analytic and ergodic-theoretic in nature.

Here is the first instance of such a challenge. In his proof for the equivalence of ergodicity and weak-mixing on $\mathbb{Z}$ and $\mathbb{R}$ \cite{podgorski}, Podgórski made the following key observation: ergodicity and weak-mixing can be tested using a certain class of positive definite functions on the group, and such functions have a certain convergence property that, in our language, is called {\it Følner convergence} (see Definition \ref{Definition: Folner Convergence}). This convergence phenomenon was used also by Rosiński \& Żak \cite{rosinskizak} in showing, more generally, that ergodicity and weak-mixing are equivalent in infinitely divisible processes on $\mathbb{Z}$ or $\mathbb{R}$. A discussion of this phenomenon for Abelian groups can be found in \cite[\textsection 8.5]{samorodnitsky2016}. The technique they used to establish the Følner convergence of positive definite functions was Bochner Theorem from classical harmonic analysis. While this technique works for Abelian groups, it is no longer applicable for non-Abelian amenable groups.

We take a new approach, considering another notion called {\it almost convergence} that was first studied in a classical work of Lorentz \cite{lorentz} for $\mathbb{Z}$ and its extension to amenable groups was studied by Day \cite{day}, Dye \cite{dye}, Douglass \cite{douglass} and many others. See Definition \ref{Definition: Almost Convergence}. Using the Ryll-Nardzewski Fixed Point Theorem \cite{ryll} it can be shown that positive definite functions on a locally compact Hausdorff amenable group are almost convergent. We then apply a criterion due to Dye and Douglass to obtain Følner convergence for positive definite functions on countable amenable groups. See Section \ref{Section: Positive Definite}.

The next challenge is purely ergodic-theoretical and it is to formulate the positive-null decomposition for systems of countable amenable groups. This was done for $\mathbb{Z}$ and $\mathbb{R}$ by Krengel \cite{krengel1967} and it was generalized to $\mathbb{Z}^{d}$ and $\mathbb{R}^{d}$ by Wang, Roy \& Stoev \cite{wangroystoev}. Here we generalize another theorem by Krengel in order to introduce a complete treatment of the positive-null decomposition of systems of countable amenable groups. See Theorems \ref{Theorem: Krengel Amenable} and \ref{Theorem: Positive-Null Decomposition}.

Having the positive-null decomposition in hands, we introduce Theorem \ref{Theorem: Weak-mixing Characterization} in which we refine Gross' condition of weak-mixing \cite{gross} in a way that is necessary to our proofs. An important feature of our formulation of the positive-null decomposition is that it is formulated concretely using Følner sequences. Thus, we are able to apply the Lindenstrauss Pointwise Ergodic Theorem \cite{lindenstrauss} to obtain an amenable group analogue to Samorodnitsky's test \cite[Theorem~2.1]{samorodnitsky2005} and its generalization by Wang, Roy \& Stoev \cite[Theorem~3.1]{wangroystoev}, for processes having either positive or null Rosiński minimal spectral representation. See Theorem \ref{Theorem: Positive Null}.

\subsection*{Notations and general settings}

Set- and group-theoretic: By writing $F\Subset G$ we mean that $F$ is a finite subset of $G$. For a finite set $F$ we write $\left|F\right|$ for its cardinality. The symmetric difference of sets $A$ and $B$ will be denoted $A\triangle B$. If $G$ is a group then for $g\in G$ and $A\subset G$ we write $gA=\left\{ gh:h\in A\right\}$, and for $A,B\subset G$ we write $AB=\left\{gh:g\in A,h\in B\right\}$. Unless stated otherwise, $G$ stands everywhere for a countably infinite amenable group with the discrete topology.

Probability- and ergodic-theoretic: For a $G$-process $\mathbf{X}$ we denote its distribution, as a probability measure on $\mathbb{R}^{G}$, by $\mathbb{P}_{\mathbf{X}}$. The relation $\ed$ stands for equality in distribution. By $\mathbb{E}\left(\cdot\right)$ we refer to the operator of expectation with respect to a probability measure that should be understood in the context. The measure spaces we consider will be denoted by $\left(\Omega,\mu\right)$, and we omit the notation for the σ-algebra. All the measurable spaces we consider are standard Borel spaces, namely the σ-algebra is the Borel σ-algebra induced from a structure of complete separable metric space. All the measures we consider are Borel and σ-finite.

To ease the orientation and as a rule of thumb, we usually use the letter $\psi$ to denote functions on the group $G$; the letter $\phi$ to denote functions of the process $\mathbf{X}$; and, the letter $f$ to denote functions on a measure space, which will be usually denoted by $\Omega$. The notation $f_0$ is always used to denote the special function in the Rosiński spectral representation of the process as in Theorem \ref{Theorem: Rosiński}.

We refer to a $G$-system as {\it null} if it has no a.c.i.p. and as {\it positive} if it has an a.c.i.p. that is mutually absolutely continuous with the given measure. We refer to a Rosiński spectral representation of a stationary $S\alpha S$ process also as {\it null} (resp. {\it positive}), if its underlying system is null (resp. positive).

For a countable set $I$ and a real-valued non-negative function $\psi$ on $I$, we write
\begin{equation}
\label{eq: countable limits}
\lim_{i\in I}\psi\left(i\right)=c
\end{equation}
for some $c\in\left[0,\infty\right)$ if, for every $\epsilon>0$, the set $\left\{ i\in I:\left|\psi\left(i\right)-c\right|>\epsilon\right\}$ is finite. Similarly, we write $\lim_{i\in I}\psi\left(i\right)=\infty$ if, for every $M>0$, the set $\left\{ i\in I:\psi\left(i\right)<M\right\}$ is finite.
 
\section{Positive definite functions on amenable groups}
\label{Section: Positive Definite}

Let $G$ be a locally compact Hausdorff (LCH) group. Let $\lambda$ and $\rho$ to be the actions of $G$ on the linear space of functions from $G$ to $\mathbb{C}$ by translation from the left and from the right, respectively. Namely, for $g\in G$ and $\psi:G\to\mathbb{C}$,
$$\lambda_{g}\psi\left(h\right)=\psi\left(gh\right)\,\,\text{and}\,\,\,\rho_{g}\psi\left(h\right)=\psi\left(hg^{-1}\right).$$

Consider the Banach algebra $L^{\infty}\left(G\right)$ (or $l^{\infty}\left(G\right)$, when $G$ is countable) of essentially bounded measurable functions from $G$ to $\mathbb{C},$ where the reference measure is the Haar measure of $G$. Consider also the Banach sub-algebra $\mathcal{C}_{b}\left(G\right)$ of continuous bounded functions on $G$. In the following we consider two important sub-algebras of $\mathcal{C}_{b}\left(G\right)$: the algebra of positive definite functions on $G$ and the algebra of weakly almost periodic functions on $G$. Let us briefly recall the fundamentals of these algebras.

A function $\psi\in\mathcal{C}_b\left(G\right)$ is said to be {\it positive definite} (sometimes also {\it of positive type}) if for every positive integer $n$ and every choice of $c_{1},\dotsc,c_{n}\in\mathbb{C}$ and $g_{1},\dotsc,g_{n}\in G$,
$$\sum_{i,j=1}^{n}c_{i}\overline{c_{j}}\psi\left(g_{j}^{-1}g_{i}\right)\geq0.$$
By the well-known GNS construction (see for instance \cite[\textsection 1.B]{bekka2020}, \cite[\textsection C]{bekka2008}), every positive definite function $\psi$ on $G$ is of the form
$$\psi\left(g\right)=\ip{\pi\left(g\right)\xi_{0}}{\xi_{0}},\quad g\in G,$$
where $\pi:G\to\mathrm{U}\left(\mathcal{H}\right)$ is a unitary representation of $G$ in a Hilbert space $\mathcal{H}$ and $\xi_{0}\in\mathcal{H}$ is a cyclic vector for $\pi$. For the fundamental properties of positive definite functions we refer to \cite[\textsection 32]{hewitt} and \cite[\textsection C.4]{bekka2008}.

A function $\psi\in\mathcal{C}_{b}\left(G\right)$ is said to be {\it weakly almost periodic} if the set $\left\{ \lambda_{g}\psi:g\in G\right\}$  forms a relatively weakly compact subset in $l^{\infty}\left(G\right)$. It turns out that, in the definition of weakly almost periodicity, there is no difference if we consider $\rho$ instead of $\lambda$. One way to produce a weakly almost periodic function on $G$ is by letting
$$\psi\left(g\right)=\ip{\pi\left(g\right)\xi_{0}}{\eta_{0}},\quad g\in G,$$
where $\pi:G\to\mathrm{U}\left(\mathcal{H}\right)$ is a unitary representation of $G$ in a Hilbert space $\mathcal{H}$ and $\xi_{0},\eta_{0}\in\mathcal{H}$ are arbitrary vectors. See for instance \cite[Corollary~4.23]{berglund}. Thus, every positive definite function is weakly almost periodic. For comprehensive accounts of weakly almost periodic functions we refer to \cite{berglund,burckel}. 

Recall that if $V$ is some closed subspace of $L^{\infty}\left(G\right)$, a {\it left-invariant mean} on $V$ is a continuous linear functional $\mathbf{m}$ defined on $V$ that satisfies $\mathbf{m}\left(1\right)=1$; $\mathbf{m}\left(\psi\right)\geq0$ for all non-negative real-valued $\psi\in V$; and $\mathbf{m}\circ\lambda_{g}=\mathbf{m}$ for all $g\in G$. The definition for right-invariant mean is analogous. By definition, an LCH group $G$ is amenable if $L^{\infty}\left(G\right)$ itself admits an invariant mean. In this case we say that $\mathbf{m}$ is an invariant mean on $G$.

The following notion was formulated for the integers in a classical work of Lorentz \cite{lorentz} and it has a natural generalization for amenable groups as follows.

\begin{definition}[Almost Convergence]
\label{Definition: Almost Convergence}
A function $\psi\in L^{\infty}\left(G\right)$ is said to be {\it almost convergent} to a constant $E_{G}\left(\psi\right)$ if for every invariant mean $\mathbf{m}$ on $G$, from the left or from the right, it holds that $\mathbf{m}\left(\psi\right)=E_{G}\left(\psi\right)$. In case that $\psi$ is almost convergent we call $E_{G}\left(\psi\right)$ the {\it universal mean} of $\psi$.
\end{definition}

It is an immediate observation that the class of almost convergent functions is a closed subspace of $L^{\infty}\left(G\right)$ containing the constant functions. Thus, the universal mean assignment $\psi\mapsto E_{G}\left(\psi\right)$ is a linear functional defined on this subspace.

\begin{proposition}
\label{Proposotion: P.D. is A.C.}
Let $G$ be an amenable LCH group. Then every weakly almost periodic function on $G$ is almost convergent.
\end{proposition}

\begin{proof}
Suppose that $\psi\in l^{\infty}\left(G\right)$ is a weakly almost periodic function. Let $K^{\lambda}_{\psi}$ be the closed convex hull of the set $\left\{\lambda_g\psi:g\in G\right\}$ in $L^{\infty}\left(G\right)$. Since $\psi$ is weakly almost periodic and by the virtue of Krein–Smulian theorem, $K^{\lambda}_{\psi}$ is weakly compact. Since $G$ acts on $K_{\psi}$ by affine isometries, $\lambda_g \psi'=\psi'\circ g$, it follows from the Ryll-Nardzewski Fixed Point Theorem \cite{ryll} that there exists an element $\psi'\in K^{\lambda}_{\psi}$ such that $\lambda_g\psi'=\psi'$ for all $g\in G$, namely $\psi'$ is a constant function, so we write $\psi'=E^{\lambda}\left(\psi\right)$ where $E^{\lambda}\left(\psi\right)$ is a constant. Observe that every left-invariant mean $\mathbf{m}$ on $G$ takes the value $\mathbf{m}\left(\psi\right)$ on all of $K^{\lambda}_{\psi}$. Then, since $K^{\lambda}_{\psi}$ contains the constant function $E^{\lambda}\left(\psi\right)$, necessarily $\mathbf{m}\left(\psi\right)=E^{\lambda}\left(\psi\right)$ for every left-invariant mean $\mathbf{m}$ on $G$.

Using the same reasoning with $K^{\rho}_{\psi}$, the closed convex hull of the set $\left\{\rho_g\psi:g\in G\right\}$ in $L^{\infty}\left(G\right)$, we obtain a constant $E^{\rho}\left(\psi\right)$ such that $\mathbf{m}\left(\psi\right)=E^{\rho}\left(\psi\right)$ for every right-invariant mean $\mathbf{m}$ on $G$.

Finally, as $G$ is amenable there is always a two-sided invariant mean on $G$ (see for instance \cite[Theorem~4.10]{kerr2016ergodic}). Alternatively, there is always a (unique) two-sided invariant mean on the algebra of weakly almost periodic functions (cf. \cite[Corollary~1.26]{burckel}, \cite[\textsection 3]{berglund}, \cite[\textsection 3.1]{greenleaf}). In any case, such a two-sided invariant mean assigns to $\psi$ both $E^{\lambda}\left(\psi\right)$ and $E^{\rho}\left(\psi\right)$ so that necessarily $E^{\lambda}\left(\psi\right)=E^{\rho}\left(\psi\right)$. Thus, $\psi$ is almost convergent.
\end{proof}

So far we discussed general amenable LCH group $G$. When $G$ is countable, it is well-known that it is amenable if and only if it admits a (left) {\it Følner sequence}, that is a sequence $\left(F_{N}\right)_{N=1}^{\infty}$ of finite subsets of $G$ satisfying the (left) {\it Følner property} by which, for every $g\in G$,
$$\lim_{N\to\infty}\frac{\left|gF_{N}\triangle F_{N}\right|}{\left|F_{N}\right|}=0.$$
Moreover, this Følner sequence can be chosen to be increasing, namely $F_{1}\subset F_{2}\subset\dotsm,$ and exhausting $G,$ namely $G=F_{1}\cup F_{2}\cup\dotsm$. Thus, whenever we refer to a Følner sequence we shall assume that it is increasing and exhaustive in addition to the Følner property. For comprehensive accounts of amenability see \cite{namioka,paterson}. Every amenable group admits a two-sided invariant mean and, in the countable case, a two-sided Følner sequence (see for instance \cite[Theorem~4.10]{kerr2016ergodic}).

In Lorentz's work on almost convergence in $\mathbb{Z}$, he showed that the universal mean of a real-valued almost convergent sequence can be computed using Cesàro summation along the sequence in a certain uniform way. When it comes to countable amenable groups, Douglass in \cite[Theorem~4.1]{douglass} used an idea of Dye \cite{dye} to generalize Lorentz's result, showing that the universal mean of an almost convergent function can be computed using two-sided Følner sequences. Before we formulate this criterion let us introduce the convenient notation, for $F\Subset G$ and a function $\psi:G\to\mathbb{C}$,
$$E_{F}\left(\psi\right)=\frac{1}{\left|F\right|}\sum_{g\in F}\psi\left(g\right).$$

\begin{definition}[Følner convergence]
\label{Definition: Folner Convergence}
Let $G$ be a countable amenable group. A function $\psi\in l^{\infty}\left(G\right)$ is said to be {\it Følner convergent} to a constant $E_{G}^{\text{Følner}}\left(\psi\right)$ if, for every two-sided Følner sequence $\left(F_{N}\right)_{N=1}^{\infty}$ for $G$,
$$\lim_{N\to\infty}\sup_{g\in G}\left|E_{F_{N}}\left(\lambda_{g}\psi\right)-E_{G}^{\text{Følner}}\left(\psi\right)\right|=\lim_{N\to\infty}\sup_{g\in G}\left|E_{F_{N}}\left(\rho_{g}\psi\right)-E_{G}^{\text{Følner}}\left(\psi\right)\right|=0.$$
In case that $\psi$ is Følner convergent we call $E_{G}^{\text{Følner}}\left(\psi\right)$ the {\it Følner mean} of $\psi$.
\end{definition}

The Dye-Douglass Criterion \cite{dye, douglass} reads as follows (see Appendix \ref{Appendix: Dye-Douglass Criterion}).

\begin{theorem}[Dye-Douglass Criterion]
\label{Theorem: Dye-Douglass Criterion}
A bounded real-valued function on a countable amenable group is almost convergent if and only if it is Følner convergent. In this case, its universal mean and its Følner mean coincide.
\end{theorem}

Combining the Dye-Douglass Criterion with Proposition \ref{Proposotion: P.D. is A.C.}, we immediately obtain the following result that motivated the current discussion.

\begin{theorem}
\label{Theorem: P.D. is F.C.}
Every real-valued weakly almost periodic function $\psi$ on a countable amenable group $G$ is Følner convergent to its universal mean $E_G\left(\psi\right)$ as in Definition \ref{Definition: Almost Convergence}.
\end{theorem}

\section{Ergodic properties of $S\alpha S$ processes}
\label{Section: Ergodic Properties}

Here we discuss ergodicity and weak-mixing of stationary $S\alpha S$ $G$-processes. We start with the preliminaries of the theory.

\subsection{Preliminaries}
\label{Section: Preliminaries}

Let $G$ be a group. For a $G$-process $\mathbf{X}=\left(X_{g}\right)_{g\in G}$ there is the associated $G$-system $G\overset{\lambda}{\curvearrowright}\left(\mathbb{R}^{G},\mathbb{P}_{\mathbf{X}}\right)$, where $\mathbb{P}_{\mathbf{X}}$ is the distribution of $\mathbf{X}$ on $\mathbb{R}^{G}$ and $\lambda$ is the action that is given, for $h\in G$ and $x=\left(x_{g}\right)_{g\in G}\in\mathbb{R}^{G}$, by
$$\lambda_{h}x=\left(x_{hg}\right)_{g\in G}.$$
A $G$-process $\mathbf{X}$ is called (left) {\it stationary} if its associated $G$-system $G\overset{\lambda}{\curvearrowright}\left(\mathbb{R}^{G},\mathbb{P}_{\mathbf{X}}\right)$ is measure preserving. That is to say,  $\mathbb{P}_{\mathbf{X}}\circ\lambda_{h}=\mathbb{P}_{\mathbf{X}}$ for every $h\in G$. For a $G$-process $\mathbf{X}=\left(X_{g}\right)_{g\in G}$ we consider the linear space
$$\mathrm{Lin}\left(\mathbf{X}\right)=\left\{ {\textstyle \sum_{g\in F}}c_{g}X_{g}:F\Subset G\text{ and }c:F\to\mathbb{C}\right\}.$$
Thus, $\mathbf{X}$ is an $S\alpha S$ $G$-process if every element of $\mathrm{Lin}\left(\mathbf{X}\right)$ follows the $S\alpha S$ distribution. In \cite[Theorem~2.1]{schilder} (see also \cite[\textsection 0]{hardin}), Schilder showed that there exists a quasi-norm $\no{\cdot}_{\alpha}$ on $\mathrm{Lin}\left(\mathbf{X}\right)$ whose metric structure induces convergence in probability, such that for every $\phi\in\mathrm{Lin}\left(\mathbf{X}\right)$,
\begin{equation}
\label{eq: alpha norm}
    \mathbb{E}\left(\exp\left(i\phi\right)\right)=\exp\left(-\no{\phi}_{\alpha}^{\alpha}\right).
\end{equation}
Thus, we denote the completion of $\mathrm{Lin}\left(\mathbf{X}\right)$ with respect to $\no{\cdot}_{\alpha}$ by $L^{\alpha}\left(\mathbf{X}\right)$.\footnote{The notation $L^{\alpha}\left(\mathbf{X}\right)$ should not be confused with the standard notations for $\alpha$-integrable functions. However, $L^{\alpha}\left(\mathbf{X}\right)$ is embedded isometrically in some space of $\alpha$-integrable functions on an abstract measure space. See \cite{schilder,hardin}.} Then $\lambda$ gives rise to an action on $L^{\alpha}\left(\mathbf{X}\right)$ by $\lambda_{g}f=f\circ\lambda_{g}$ and, when $\mathbf{X}$ is stationary, it is an action by isometries. The following is a fundamental observation made by Podgórski \cite{podgorski}.

\begin{proposition}
\label{Proposition: Is Positive Definite}
Let $\mathbf{X}$ be a stationary $S\alpha S$ $G$-process. Then for every $\phi\in L^{\alpha}\left(\mathbf{X}\right),$ the function $\psi_{\phi}:G\to\mathbb{R}$ given by
$$\psi_{\phi}\left(g\right)=\exp\left(-\no{\lambda_{g}\phi-\phi}_{\alpha}^{\alpha}\right),\quad g\in G,$$
is positive definite.
\end{proposition}

We remark that in the Abelian case, a similar phenomenon is true in a more general setting. See \cite[\textsection 3.6]{samorodnitsky2016}).

\begin{proof}
Let $\phi\in L^{\alpha}\left(\mathbf{X}\right)$. Using the stationarity of $\mathbf{X}$ and Identity \eqref{eq: alpha norm}, for every $g,h\in G$,
$$\psi_{\phi}\left(g^{-1}h\right) = \exp\left(-\no{\lambda_{h}\phi-\lambda_{g}\phi}_{\alpha}^{\alpha}\right) = \mathbb{E}\left(\exp\left(i\lambda_{h}\phi\right)\overline{\exp\left(i\lambda_{g}\phi\right)}\right).$$
Thus, for every positive integer $n$ and every choice of $c_{1},\dotsc,c_{n}\in\mathbb{C}$ and $g_{1},\dotsc,g_{n}\in G$,
\begin{equation}
\sum_{i,j=1}^{n}c_{i}\overline{c_{j}}\psi_{\phi}\left(g_{j}^{-1}g_{i}\right)=\mathbb{E}\left(\left|\sum_{i=1}^{n}c_{i}\exp\left(i\lambda_{g_{i}}\phi\right)\right|^{2}\right)\geq0.\qedhere\nonumber
\end{equation}
\end{proof}

Next we record the fundamental Rosiński minimal spectral representation of stationary non-Gaussian $S\alpha S$ $G$-processes. We follow some of the definitions and notations that already established in the introduction.

\begin{definition}[Minimal Spectral Representation \cite{hardin}]
\label{Definition: Spectral Representation}
Let $G$ be a countable group and $\mathbf{X}=\left(X_{g}\right)_{g\in G}$ a stationary non-Gaussian $S\alpha S$ $G$-process. A {\it spectral representation} for $\mathbf{X}$ is a $G$-system $G\overset{\f a}{\curvearrowright}\left(\Omega,\mu\right)$, a special function $f_{0}\in L^{\alpha}\left(\mu\right)$ and an action $G\overset{\f u}{\curvearrowright}L^{\alpha}\left(\mu\right)$ by isometries such that
$$\mathbf{X}\ed\left(\int_{\Omega}\f u_{g}f_0dM_{\alpha}\right)_{g\in G}.$$
From the properties of random measures, for all $\sum_{g\in F}c_{g}X_{g}\in\mathrm{Lin}\left(\mathbf{X}\right)$,
$$\mathbb{E}\left(\exp\left(i{\textstyle \sum_{g\in F}}c_{g}X_{g}\right)\right)=\exp\left(-\no{{\textstyle \sum_{g\in F}}c_{g}\f u_{g}f_{0}}_{\alpha}^{\alpha}\right).$$
Such a spectral representation is said to be {\it minimal} if, up to a $\mu$-null set,
\begin{enumerate}
\item $\Omega$ is equal to the common support of the functions $\left\{ \f u_{g}f_{0}:g\in G\right\}$, and
\item the σ-algebra on $\Omega$ is generated by the (possibly infinite-valued) functions $\left\{ \f u_{g}f_{0}/\f u_{h}f_{0}:g,h\in G\right\}$.
\end{enumerate}
\end{definition}

Recall that a {\it cocycle} for a $G$-system $G\overset{\f a}{\curvearrowright}\left(\Omega,\mu\right)$ with values in some Abelian group $A$, is a function $c:G\times\Omega\to A$ such that, for all $g,h\in G$,
$$c_{gh}=c_{g}\circ\f a_{h}\cdot c_{h}\quad\text{almost surely}.$$

\begin{theorem}[Rosiński Minimal Spectral Representation \cite{rosinski1995}]
\label{Theorem: Rosiński}
Let $G$ be a countable group and $\mathbf{X}=\left(X_{g}\right)_{g\in G}$ a stationary non-Gaussian $S\alpha S$ $G$-process. There exists a $G$-system $G\overset{a}{\curvearrowright}\left(\Omega,\mu\right)$, a cocycle $c:G\times\Omega\to\left\{ +1,-1\right\}$ and a special function $f_0\in L^{\alpha}\left(\mu\right)$, such that the minimal spectral representation of $\mathbf{X}$ takes the form
$$\f u_{g}f_0=c_{g}\sqrt[\alpha]{\frac{d\mu\circ\f a_{g}}{d\mu}}f_0\circ\f a_{g},\quad g\in G.$$
Moreover, the minimal spectral representation is unique up to isomorphism of G-systems.
\end{theorem}

We mention that Rosiński in \cite[Theorem~3.1]{rosinski1995} (see also \cite[Remark~2.3]{rosinski1995}) showed this for $\mathbb{Z}$, and its proof can be adapted to general countable groups. See \cite[\textsection 2]{sarkaroy} and \cite[\textsection 4.3]{roy2020}. As for the uniqueness and a discussion of the spectral representation on countable groups we refer to \cite[Sections~4, 5]{roy2020}.

\subsection{Ergodicity}

Suppose that $G$ is a countable amenable group and $\mathbf{X}$ is a stationary $G$-process. Let $L^{2}\left(\mathbf{X}\right)$ be the Hilbert space of square-integrable functions on $\mathbf{X}$ with the standard inner product $\ip{\phi_{0}}{\phi_{1}}=\mathbb{E}\left(\phi_{0}\overline{\phi}_{1}\right)$. For $F\Subset G$ let $S_{F}:L^{2}\left(\mathbf{X}\right)\to L^{2}\left(\mathbf{X}\right)$ be the operator
$$S_{F}\phi=\frac{1}{\left|F\right|}\sum_{g\in F}\lambda_{g}\phi.$$
By the von Neumann Ergodic Theorem for amenable groups (see for instance \cite[Theorem~2.1]{weiss}), for every Følner sequence $\left(F_{N}\right)_{N=1}^{\infty}$ for $G$ and $\phi\in L^{2}\left(\mathbf{X}\right)$,
$$\lim_{N\to\infty}S_{F_{N}}\phi=\mathbb{E}\left(\phi\mid\mathrm{Inv}\left(\mathbf{X}\right)\right)\text{ in }L^{2}\left(\mathbf{X}\right).$$
Here $\mathrm{Inv}\left(\mathbf{X}\right)$ is the σ-algebra of events in $\mathbf{X}$ that are invariant to $\lambda$, and $\mathbb{E}\left(\cdot\mid\mathrm{Inv}\left(\mathbf{X}\right)\right)$ is the conditional expectation from $L^{2}\left(\mathbf{X}\right)$ to its subspace of $\mathrm{Inv}\left(\mathbf{X}\right)$-measurable functions. In view of the von Neumann Ergodic Theorem, ergodicity possesses the following different equivalent formulations (cf. \cite[\textsection 2.4]{petersen}).

\begin{definition}[Ergodicity]
\label{Definition: Ergodicity}
A stationary $G$-process $\mathbf{X}$ is called {\it ergodic} if the following equivalent properties hold.
\begin{enumerate}
\item $\mathrm{Inv}\left(\mathbf{X}\right)$ is the trivial σ-algebra modulo events of zero probability;
\item the limit in the von Neumann Ergodic Theorem is constant and equals $\mathbb{E}\left(\phi\right)$ for all $\phi\in L^{2}\left(\mathbf{X}\right)$;
\item for every (equivalently, there exists a) Følner sequence $\left(F_{N}\right)_{N=1}^{\infty}$ for $G$ and every pair of events $A,B$ in $\mathbf{X}$,
$$\lim_{N\to\infty}{\textstyle \frac{1}{\left|F_{N}\right|}\sum_{g\in F_{N}}}\mathbb{P}_{\mathbf{X}}\left(A\cap\lambda_{g}\left(B\right)\right)=\mathbb{P}_{\mathbf{X}}\left(A\right)\mathbb{P}_{\mathbf{X}}\left(B\right).$$
\end{enumerate}
\end{definition}

The spectral structure of $S\alpha S$ processes allows us to provide a more explicit characterization of ergodicity. The following proposition was proved in the course of the proof of \cite[Theorem~2]{podgorski}. It uses the standard fact, due to the Stone--Weierstrass Theorem, that the linear span of $\left\{ \exp\left(i\phi\right):\phi\in\mathrm{Lin}\left(\mathbf{X}\right)\right\}$ is dense in $L^2\left(\mathbf{X}\right)$ and, based on Identity \eqref{eq: alpha norm}, it holds for all countable amenable groups with some obvious modifications.

\begin{proposition}
\label{Proposition: Ergodicity}
A stationary $S\alpha S$ $G$-process $\mathbf{X}$ is ergodic if and only if for every Følner sequence $\left(F_{N}\right)_{N=1}^{\infty}$ for $G$ and every $\phi\in L^{\alpha}\left(\mathbf{X}\right)$,
$$\lim_{N\to\infty}\mathbb{E}\left(\left|S_{F_{N}}\exp\left(i\phi\right)\right|^{2}\right)=\exp\left(-2\no{\phi}_{\alpha}^{\alpha}\right).$$
\end{proposition}

The following theorem is the main result of this section and is the reason for us to establish Theorem \ref{Theorem: P.D. is F.C.}.

\begin{theorem}
\label{Theorem: Ergodicity}
Let $\mathbf{X}$ be a stationary $S\alpha S$ $G$-process. For every two-sided Følner sequence $\left(F_{N}\right)_{N=1}^{\infty}$ for $G$ and every $\phi\in L^{\alpha}\left(\mathbf{X}\right)$,
$$\lim_{N\to\infty}{\textstyle \frac{1}{\left|F_{N}\right|}\sum_{g\in F_{N}}}\exp\left(-\no{\lambda_{g}\phi-\phi}_{\alpha}^{\alpha}\right)=\lim_{N\to\infty}\mathbb{E}\left(\left|S_{F_{N}}\exp\left(i\phi\right)\right|^{2}\right).$$
\end{theorem}

In view of Proposition \ref{Proposition: Ergodicity} we get the following immediate corollary:

\begin{corollary}
\label{Corollary: Ergodicity}
A stationary $S\alpha S$ $G$-process $\mathbf{X}$ is ergodic if and only if for every two-sided Følner sequence $\left(F_{N}\right)_{N=1}^{\infty}$ for $G$ and every $\phi\in L^{\alpha}\left(\mathbf{X}\right)$,
$$\lim_{N\to\infty}{\textstyle \frac{1}{\left|F_{N}\right|}\sum_{g\in F_{N}}}\exp\left(-\no{\lambda_{g}\phi-\phi}_{\alpha}^{\alpha}\right)=\exp\left(-2\no{\phi}_{\alpha}^{\alpha}\right).$$
\end{corollary}

\begin{proof}[Proof of Theorem \ref{Theorem: Ergodicity}]
Let $\phi\in L^{\alpha}\left(\mathbf{X}\right)$ and define $\psi_{\phi}:G\to\mathbb{R}$ by
$$\psi_{\phi}\left(g\right)=\exp\left(-\no{\lambda_{g}\phi-\phi}_{\alpha}^{\alpha}\right),\quad g\in G.$$
Let $\left(F_{N}\right)_{N=1}^{\infty}$ be a two-sided Følner sequence for $G$ and we need to show that
$$\lim_{N\to\infty}{\textstyle \frac{1}{\left|F_{N}\right|}\sum_{g\in F_{N}}}\psi_{\phi}\left(g\right)=\lim_{N\to\infty}\mathbb{E}\left(\left|S_{F_{N}}\exp\left(i\phi\right)\right|^{2}\right).$$
By Proposition \ref{Proposition: Is Positive Definite} and Theorem \ref{Theorem: P.D. is F.C.}, $\psi_{\phi}$ is Følner convergent to $E_{G}\left(\psi_{\phi}\right)$. We then need to show that
$$\lim_{N\to\infty}\mathbb{E}\left(\left|S_{F_{N}}\exp\left(i\phi\right)\right|^{2}\right)=E_{G}\left(\psi_{\phi}\right).$$
From the stationarity of $\mathbf{X}$ and the identity \eqref{eq: alpha norm}, for every $N$ we may write
\begin{align*}
    \mathbb{E}\left(\left|S_{F_{N}}\exp\left(i\phi\right)\right|^{2}\right)
    & = {\textstyle \frac{1}{\left|F_{N}\right|^{2}}\sum_{g,h\in F_{N}}}\mathbb{E}\left(\exp\left(i\left(\lambda_{g}\phi-\lambda_{h}\phi\right)\right)\right)\\
    & = {\textstyle \frac{1}{\left|F_{N}\right|^{2}}\sum_{g,h\in F_{N}}}\psi_{\phi}\left(h^{-1}g\right)={\textstyle \frac{1}{\left|F_{N}\right|^{2}}\sum_{g,h\in F_{N}}}\lambda_{h}\psi_{\phi}\left(g\right)\\
    & = {\textstyle \frac{1}{\left|F_{N}\right|}\sum_{h\in F_{N}}}\left({\textstyle \frac{1}{\left|F_{N}\right|}\sum_{g\in F_{N}}}\lambda_{h}\psi_{\phi}\left(g\right)\right)={\textstyle \frac{1}{\left|F_{N}\right|}\sum_{h\in F_{N}}}E_{F_{N}}\left(\lambda_{h}\psi_{\phi}\right).
\end{align*}
By the definition of Følner convergence, for an arbitrary $\epsilon>0$, if $N$ is sufficiently large then
$$\sup_{h\in G}\left|E_{F_{N}}\left(\lambda_{h}\psi_{\phi}\right)-E_{G}\left(\psi_{\phi}\right)\right|<\epsilon.$$
It then follows that
$$\left|\mathbb{E}\left(\left|S_{F_{N}}\exp\left(i\phi\right)\right|^{2}\right)-E_{G}\left(\psi_{\phi}\right)\right|	\leq{\textstyle \frac{1}{\left|F_{N}\right|}\sum_{h\in F_{N}}}\left|E_{G}\left(\lambda_{h}\psi_{\phi}\right)-E_{G}\left(\psi_{\phi}\right)\right|<\epsilon.$$
As $\epsilon>0$ is arbitrary the proof is complete.
\end{proof}

\subsection{Weak-mixing}
\label{Subsection: Weak-Mixing}

The following different formulations of weak-mixing are well-known to be equivalent. It is also well-known that in order to obtain weak-mixing it is enough to verify each of the following formulations on one Følner sequence (cf. \cite[\textsection 2.6]{petersen}).

\begin{definition}[Weak-mixing]
\label{Definition: Weak-mixing}
A stationary $G$-process $\mathbf{X}$ is called {\it weakly-mixing} if the following equivalent properties hold for every Følner sequence $\left(F_{N}\right)_{N=1}^{\infty}$ for $G$.
\begin{enumerate}
\item For every pair of events $A,B$ in $\mathbf{X}$,
$$\lim_{N\to\infty}{\textstyle \frac{1}{\left|F_{N}\right|}\sum_{g\in F_{N}}}\left|\mathbb{P}_{\mathbf{X}}\left(A\cap\lambda_{g}\left(B\right)\right)-\mathbb{P}_{\mathbf{X}}\left(A\right)\mathbb{P}_{\mathbf{X}}\left(B\right)\right|=0;$$
\item for every $\phi_{0},\phi_{1}\in L^{2}\left(\mathbf{X}\right)$,
$$\lim_{N\to\infty}{\textstyle \frac{1}{\left|F_{N}\right|}\sum_{g\in F_{N}}}\left|\mathbb{E}\left(\phi_{0}\overline{\lambda_{g}\phi_{1}}\right)-\mathbb{E}\left(\phi_{0}\right)\overline{\mathbb{E}\left(\phi_{1}\right)}\right|=0;$$
\item for every $\phi\in L^{2}\left(\mathbf{X}\right)$,
$$\lim_{N\to\infty}{\textstyle \frac{1}{\left|F_{N}\right|}\sum_{g\in F_{N}}}\left|\mathbb{E}\left(\phi\overline{\lambda_{g}\phi}\right)-\left|\mathbb{E}\left(\phi\right)\right|^{2}\right|=0.$$
\end{enumerate}
\end{definition}

A key lemma that we are about to use is the following real analysis lemma due to Podgórski \& Weron. While the lemma was formulated for $\mathbb{N}$, the proof appears in \cite{podwer} remains valid for every countable set $G$ with an increasing sequence of finite sets $\left(F_{N}\right)_{N=1}^{\infty}$, and the Følner property is not needed.

\begin{lemma}[Podgórski \& Weron]
\label{Lemma: Podgórski-Weron}
Let $\left(F_{N}\right)_{N=1}^{\infty}$ be an increasing sequence of finite sets and let $G=F_{1}\cup F_{2}\cup\dotsm$. For a real-valued bounded function $\psi:G\to\mathbb{R}$ the following are equivalent.
\begin{enumerate}
\item The limit $\lim_{N\to\infty}\frac{1}{\left|F_{N}\right|}\sum_{g\in F_{N}}\left|\exp\left(c\psi\left(g\right)\right)-1\right|=0$ holds for all $c>0$.
\item The limit $\lim_{N\to\infty}\frac{1}{\left|F_{N}\right|}\sum_{g\in F_{N}}\exp\left(c\psi\left(g\right)\right)=1$ holds for all $c>0$.
\end{enumerate}
\end{lemma}

We are now ready to prove Theorem \ref{Theorem: Ergodicity iff Weak-Mixing}.

\begin{proof}[Proof of Theorem \ref{Theorem: Ergodicity iff Weak-Mixing}]
It is obvious that weak-mixing implies ergodicity. Suppose that $\mathbf{X}$ is ergodic, so that by Corollary \ref{Corollary: Ergodicity} for every two-sided Følner sequence $\left(F_{N}\right)_{N=1}^{\infty}$ for $G$ and $\phi\in L^{\alpha}\left(\mathbf{X}\right)$,
\begin{equation}
\label{eq:ErgodicW-M}
\lim_{N\to\infty}{\textstyle \frac{1}{\left|F_{N}\right|}\sum_{g\in F_{N}}}\exp\left(2\no{\phi}_{\alpha}^{\alpha}-\no{\lambda_{g}\phi-\phi}_{\alpha}^{\alpha}\right)=1.
\end{equation}
For $\phi\in L^{\alpha}\left(\mathbf{X}\right)$ let $\psi_{\phi}:G\to\mathbb{R}$ be the function
$$\psi_{\phi}\left(g\right)=2\no{\phi}_{\alpha}^{\alpha}-\no{\lambda_{g}\phi-\phi}_{\alpha}^{\alpha}.$$
Since $\psi_{c^{1/\alpha}\phi}=c\psi_{\phi}$ for every $c>0$, applying Corollary \ref{Corollary: Ergodicity} and the limit \eqref{eq:ErgodicW-M} we obtain that
$$\lim_{N\to\infty}E_{F_{N}}\left(\exp\left(c\psi_{\phi}\right)\right)=1\text{ for all }c>0.$$
By Lemma \ref{Lemma: Podgórski-Weron} it follows that, in particular,
$$\lim_{N\to\infty}E_{F_{N}}\left|\exp\left(\psi_{\phi}\right)-1\right|=0.$$
But for every $g\in G$ we have
\begin{align*}
\left|\exp\left(\psi_{\phi}\left(g\right)\right)-1\right|
& = \exp\left(2\no{\phi}_{\alpha}^{\alpha}\right)\left|\exp\left(-\no{\lambda_{g}\phi-\phi}_{\alpha}^{\alpha}\right)-\exp\left(-2\no{\phi}_{\alpha}^{\alpha}\right)\right|\\
& = \exp\left(2\no{\phi}_{\alpha}^{\alpha}\right)\left|\mathbb{E}\left(\exp\left(i\left(\lambda_{g}\phi-\phi\right)\right)\right)-\left|\mathbb{E}\left(\exp\left(i\phi\right)\right)\right|^{2}\right|.
\end{align*}
It then follows that
$$\lim_{N\to\infty}{\textstyle \frac{1}{\left|F_{N}\right|}\sum_{g\in F_{N}}}\left|\mathbb{E}\left(\exp\left(i\left(\lambda_{g}\phi-\phi\right)\right)\right)-\left|\mathbb{E}\left(\exp\left(i\phi\right)\right)\right|^{2}\right|=0.$$
This establishes the weak-mixing property for the functions $\exp\left(i\phi\right)$ with $\phi\in \mathrm{Lin}\left(\mathbf{X}\right)$. By a standard approximation argument using the Stone--Weierstrass Theorem this holds for all $\phi\in L^{2}\left(\mathbf{X}\right)$ so that $\mathbf{X}$ is weakly-mixing.
\end{proof}

\section{Positive-null decomposition in amenable groups}
\label{Section: Positive-Null Decomposition}

Here we show that an analogue of Krengel's description of the positive-null decomposition for $\mathbb{Z}$ as in \cite[\textsection 3.4]{krengel1985} (cf. \cite[Theorem~1.4.4]{aaronson}), holds in countable amenable groups.

\begin{remark}
After the first version of this paper was published I found that the positive-null decomposition for amenable groups was proved by Grabarnik \& Hrushovski \cite{grabarnikhrushovski1995}. Their proof uses the machinery of weakly-wandering functions, similarly to Krengel \cite{krengel1985} and Hajian \& Kakutani \cite{hajiankakutani}. The following proof is different, and it provides a characterization of the positive-null decomposition using Følner sequences. This characterization will be important to our later proofs.
\end{remark}

Let $\left(\Omega,\mu\right)$ be a standard σ-finite measure space. Recall that a sequence of measurable functions $\left(f_{N}\right)_{N=1}^{\infty}$ is said to be {\it convergent in measure}, or {\it convergent stochastically}, to a measurable function $f$ if, for every Borel set $E\subset\Omega$ with $\mu\left(E\right)<\infty$ and every $\epsilon>0$,
$$\lim_{N\to\infty}\mu\left(\omega\in E:\left|f_{N}\left(\omega\right)-f\left(\omega\right)\right|>\epsilon\right)=0.$$
See \cite[Definition~4.8]{krengel1985}. Sometimes we denote convergence in measure by $f_{N}\xrightarrow[N\to\infty]{\mu}f$. It is an elementary fact that $f_{N}\xrightarrow[N\to\infty]{\mu}f$ if and only if every subsequence of $\left(f_{N}\right)_{N=1}^{\infty}$ has a further subsequence that converges to $f$ almost surely with respect to $\mu$. In particular, convergence in measure is preserved when passing to an absolutely continuous measure.

Let $G\overset{\f a}{\curvearrowright}\left(\Omega,\mu\right)$ be a $G$-system. As the measure $\mu$ is understood, we abbreviate the Radon-Nikodym cocycle associated to $\f a$ by
$$\f a_{g}'\left(\omega\right)=\frac{d\mu\circ\f a_{g}}{d\mu}\left(\omega\right)\text{ for }g\in G.$$
Denote by $G\overset{\widehat{\f a}}{\curvearrowright}L^{1}\left(\Omega,\mu\right)$ the {\it dual action} of $\f a$ defined by
\begin{equation}
\label{eq: dual action}
\widehat{\mathfrak{a}}_{g}f\left(\omega\right)=\f a_{g}'\left(\omega\right)f\circ\mathfrak{a}_{g}\left(\omega\right),\quad g\in G,\quad f\in L^{1}\left(\mu\right).
\end{equation}
Note that $\widehat{\f a}$ is an anti-action by isometries. It is dual to the Koopman action in the sense that, for every $f\in L^{1}\left(\mu\right)$ and $b\in L^{\infty}\left(\mu\right)$,
$$\int_{\Omega}\widehat{\f a}_{g}\left(f\right)bd\mu=\int_{\Omega}fb\circ\f a_{g}d\mu,\quad g\in G.$$

Recall that an {\it a.c.i.p} for the $G$-system $G\overset{\f a}{\curvearrowright}\left(\Omega,\mu\right)$ is a Borel probability measure on $\Omega$ that is both: absolutely continuous with respect to $\mu$ and is $\f a$-invariant. A $G$-system is said to be {\it null} if it has no a.c.i.p. There is a natural correspondence between a.c.i.p.'s for $G\overset{\f a}{\curvearrowright}\left(\Omega,\mu\right)$ and real-valued non-negative functions $f\in L^{1}\left(\mu\right)$ with $\no f_{L^{1}\left(\mu\right)}=1$ that are $\widehat{\f a}$-invariant, namely $\widehat{\f a}_{g}f=f$ for all $g\in G$. This correspondence is given by identifying an a.c.i.p. $\nu$ with the Radon-Nikodym derivative $d\nu/d\mu\in L^1\left(\Omega,\mu\right)$. We now establish a Krengel-type criterion for the existence of a.c.i.p.

\begin{theorem}
\label{Theorem: Krengel Amenable}
Let $G$ be a countable amenable group and $G\overset{\f a}{\curvearrowright}\left(\Omega,\mu\right)$ a non-singular $G$-system such that $\mu$ is a probability measure. Then a necessary and sufficient condition for the system {\rm not} to have an a.c.i.p. is that, for every Følner sequence $\left(F_{N}\right)_{N=1}^{\infty}$ for $G$ and every $f\in L^{1}\left(\mu\right)$,
$${\textstyle \frac{1}{\left|F_{N}\right|}\sum_{g\in F_{N}}}\widehat{\mathfrak{a}}_{g}f\xrightarrow[N\to\infty]{}0\text{ in measure}.$$
\end{theorem}

The following proof of Theorem \ref{Theorem: Krengel Amenable} is inspired by the arguments in the proof of Aaronson to Krengel's theorem for $\mathbb{Z}$ in \cite[Theorem~1.4.4]{aaronson}. The property of $L^2$ that is presented in \cite[Lemma~1.4.5]{aaronson} will be replaced, as suggested in \cite[\textsection 1.4]{aaronson}, by the Komlós Theorem \cite[Theorem~1]{komlos} that is applied to $L^1$.

For the ease of notations let us say that a sequence $\left(x_{n}\right)_{n=1}^{\infty}$ is {\it Cesàro convergent} to $x$ if the sequence $\left(\frac{1}{N}\sum_{n=1}^{N}x_{n}\right)_{N=1}^{\infty}$ is convergent to $x$ as $N\to\infty$, and in each instance of this terminology we specify the mode of convergence we refer to.

\begin{proof}[Proof of Theorem \ref{Theorem: Krengel Amenable}]
One implication is elementary: if $\nu$ is an a.c.i.p. then $f:=d\nu/d\mu\in L^{1}\left(\mu\right)$ satisfies $\no f_{L^{1}\left(\mu\right)}=1$ and is $\widehat{\f a}$-invariant, so that its averages along every Følner sequence can not converge to $0$ in measure.

Let us show the converse. Suppose that there is no a.c.i.p., let $f\in L^{1}\left(\mu\right)$ and some Følner sequence $\left(F_{N}\right)_{N=1}^{\infty}$ for $G$. We then need to show that
$$A_{N}f:={\textstyle \frac{1}{\left|F_{N}\right|}\sum_{g\in F_{N}}}\widehat{\f a}_{g}f\xrightarrow[N\to\infty]{\mu}0.$$
Since $\left|A_{N}f\right|\leq A_{N}\left|f\right|$ we may replace $f$ by $\left|f\right|$, so we assume without loss of generality that $f$ is real-valued and non-negative. Let $\left(A_{N_{n}}f\right)_{k=1}^{\infty}$ be an arbitrary subsequence and we show that it has a further subsequence converging to $0$ almost surely. Note that, since $\widehat{\f a}$ is an action by isometries of $L^{1}\left(\mu\right)$,
$$\no{A_{N}f}_{L^{1}\left(\mu\right)}\leq\no f_{L^{1}\left(\mu\right)}\text{ for all }N.$$
Then, as $\mu$ is a probability measure, by the Komlós Theorem \cite[Theorem~1]{komlos} there is a subsequence of $\left(A_{N_{n}}f\right)_{n=1}^{\infty}$ that is Cesàro convergent almost surely. Denote this subsequence by $\left(f_{k}\right)_{k=1}^{\infty}$ where $f_{k}:=A_{N_{n_{k}}}f$, and denote its Cesàro almost sure limit by $\overline{f}\in L^{1}\left(\mu\right)$. Note that since $f$ is real-valued and non-negative, then so is $\overline{f}$.

We claim that $\overline{f}$ is $\widehat{\f a}$-invariant. In order to see that, let $h\in G$ be arbitrary and for every $K$ we can bound
\begin{equation}
\label{eq: a-invariant}
\left|\overline{f}-\widehat{\f a}_{h}\overline{f}\right|\leq\left|\overline{f}-{\textstyle \frac{1}{K}\sum_{k=1}^{K}}f_{k}\right|+\left|{\textstyle \frac{1}{K}\sum_{k=1}^{K}}\left(f_{k}-\widehat{\f a}_{h}f_{k}\right)\right|+\left|{\textstyle \frac{1}{K}\sum_{k=1}^{K}}\widehat{\f a}_{h}f_{k}-\widehat{\f a}_{h}\overline{f}\right|.
\end{equation}
The first term vanishes as $K\to\infty$ almost surely by the definition of $\overline{f}$. The last term is proportional, almost surely, to
$$\left|{\textstyle \frac{1}{K}\sum_{k=1}^{K}}f_{k}\circ\f a_{h}-\overline{f}\circ\f a_{h}\right|.$$
Hence, by the non-singularity of $\f a_h$, it also vanishes as $K\to\infty$ almost surely by the definition of $\overline{f}$. We show that the middle term vanishes almost surely along a subsequence $\left(K_L\right)_{L=1}^{\infty}$. Considering the $L^1\left(\mu\right)$-norm of the middle term of \eqref{eq: a-invariant}, by the Følner property we have
\begin{align*}
    \no{f_{k}-\widehat{\f a}_{h}f_{k}}_{L^{1}\left(\mu\right)}
    & = {\textstyle \frac{1}{\left|F_{N_{n_{k}}}\right|}}\no{{\textstyle \sum_{g\in F_{N_{n_{k}}}}}\widehat{\f a}_{g}f-{\textstyle \sum_{g\in hF_{N_{n_{k}}}}}\widehat{\f a}_{g}f}_{L^{1}\left(\mu\right)} \\
    & = {\textstyle \frac{1}{\left|F_{N_{n_{k}}}\right|}}\no{{\textstyle \sum_{g\in F_{N_{n_{k}}}\backslash hF_{N_{n_{k}}}}}\widehat{\f a}_{g}f-{\textstyle \sum_{g\in hF_{N_{n_{k}}}\backslash F_{N_{n_{k}}}}}\widehat{\f a}_{g}f}_{L^{1}\left(\mu\right)} \\
    & \leq {\textstyle \frac{\left|hF_{N_{n_{k}}}\triangle F_{N_{n_{k}}}\right|}{\left|F_{N_{n_{k}}}\right|}}\no f_{L^{1}\left(\mu\right)}\xrightarrow[k\to\infty]{}0.
\end{align*}
It follows by the Cesàro summation of real numbers that
$$\no{{\textstyle \frac{1}{K}\sum_{k=1}^{K}}\left(f_{k}-\widehat{\f a}_{h}f_{k}\right)}_{L^{1}\left(\mu\right)}\leq{\textstyle \frac{1}{K}\sum_{k=1}^{K}}\no{f_{k}-\widehat{\f a}_{h}f_{k}}_{L^{1}\left(\mu\right)}\xrightarrow[K\to\infty]{}0.$$
As $\mu$ is a probability measure, there can be found a subsequence $\left(K_L\right)_{L=1}^{\infty}$ of the positive integers such that
$${\textstyle \frac{1}{K_{L}}\sum_{k=1}^{K_{L}}}\left(f_{k}-\widehat{\f a}_{h}f_{k}\right)\xrightarrow[L\to\infty]{}0\text{ almost surely.}$$
Thus, from the bound \eqref{eq: a-invariant} we obtained that $\overline{f}=\widehat{\f a}_{h}\overline{f}$ almost surely with an arbitrary $h$, showing that $\overline{f}$ is $\widehat{\f a}$-invariant. Since $\overline{f}$ is real-valued and non-negative, the assumption that there is no a.c.i.p. implies that necessarily $\overline{f}=0$, as otherwise it would be, up to a normalization by a positive constant, the Radon-Nikodym derivative of an a.c.i.p. Thus, $\left(f_{k}\right)_{k=1}^{\infty}$ is a sequence of real-valued non-negative functions that is Cesàro convergent to $0$ almost surely and, by the lemma appears in Appendix \ref{Appendix: Lemma on Almost-Sure}, it has a subsequence converging to $0$ almost surely. This establishes the desired convergence in measure.
\end{proof}

From Theorem \ref{Theorem: Krengel Amenable} we obtain the following positive-null decomposition. Note that while Theorem \ref{Theorem: Krengel Amenable} is restricted to $G$-systems with probability measures, the following positive-null decomposition holds for general $G$-system with σ-finite measure.

\begin{theorem}[Positive-Null Decomposition]
\label{Theorem: Positive-Null Decomposition}
Let $G$ be a countable amenable group and $G\overset{\f a}{\curvearrowright}\left(\Omega,\mu\right)$ a $G$-system. There is a partition
$$\Omega=\mathcal{P}\cup\mathcal{N}$$
of $\Omega$ into disjoint $\f a$-invariant sets $\mathcal{P}$ and $\mathcal{N}$ such that the following properties hold.
\begin{enumerate}
\item There is a special real-valued non-negative $\widehat{\f a}$-invariant function $F\in L^{1}\left(\mu\right)$ such that $\mathcal{P}=\left\{ F>0\right\}$ and $\mathcal{N}=\left\{ F=0\right\}$;
\item the system $G\overset{\f a}{\curvearrowright}\left(\mathcal{P},\mu\mid_{\mathcal{P}}\right)$ has the a.c.i.p. $F\left(\omega\right)d\mu\left(\omega\right)$ which is mutually absolutely continuous with $\mu\mid_{\mathcal{P}}$;
\item the system $G\overset{\f a}{\curvearrowright}\left(\mathcal{N},\mu\mid_{\mathcal{N}}\right)$ has no a.c.i.p.;
\item for every Følner sequence $\left(F_{N}\right)_{N=1}^{\infty}$ for $G$ and every $f\in L^{1}\left(\mu\right)$,
$$\frac{1}{\left|F_{N}\right|}\sum_{g\in F_{N}}\widehat{\f a}_{g}f\xrightarrow[N\to\infty]{}0\text{ in measure on }\mathcal{N}.$$
\end{enumerate}
\end{theorem}

When $G\overset{\f a}{\curvearrowright}\left(\Omega,\mu\right)$ is a non-singular $G$-action and $\mu$ is a probability measure so that Theorem \ref{Theorem: Positive-Null Decomposition} is applicable, there is a standard way to obtain positive-null decomposition from Theorem \ref{Theorem: Positive-Null Decomposition} using the very same argument as in \cite[Theorem~1.4.6.]{aaronson}. Note that Hopf's conservative-dissipative decomposition that is being used in this proof does hold for general countable groups. See \cite[Section~1.6]{aaronson}. In order to establish Theorem \ref{Theorem: Positive-Null Decomposition} also for σ-finite spaces we use the following proposition, and this will complete the proof of Theorem \ref{Theorem: Positive-Null Decomposition}.

\begin{proposition}
If the positive-null decomposition as in Theorem \ref{Theorem: Positive-Null Decomposition} holds for probability spaces, then it also holds for σ-finite measure spaces.
\end{proposition}

\begin{proof}
Let $G\overset{\f a}{\curvearrowright}\left(\Omega,\mu\right)$ be a non-singular $G$-system. Pick some Borel probability measure $\mu_{1}$ on $\Omega$ that is mutually absolutely continuous with $\mu$. Consider the linear operator $\mathfrak{D}:L^{1}\left(\mu_{1}\right)\to L^{1}\left(\mu\right)$ given by
$$\mathfrak{D}\left(f\right)=\frac{d\mu_{1}}{d\mu}f,\quad f\in L^{1}\left(\mu_{1}\right).$$
This is an isometric isomorphism of Banach spaces whose inverse is given by
$$\mathfrak{D}^{-1}\left(f\right)=\left(\frac{d\mu_{1}}{d\mu}\right)^{-1}f,\quad f\in L^{1}\left(\mu\right).$$
Let $\widehat{\f a}^{\left(1\right)}$ be the dual operator of $\f a$ with respect to $\mu_{1}$. Observe that $\widehat{\f a}$ and $\widehat{\f a}^{\left(1\right)}$ are conjugated by $\mathfrak{D}$, namely
\begin{equation}
\label{eq: DaaD}
\widehat{\f a}\circ\mathfrak{D}=\mathfrak{D}\circ\widehat{\f a}^{\left(1\right)}.
\end{equation}
By the assumption there is positive-null decomposition of $G\overset{\f a}{\curvearrowright}\left(\Omega,\mu_{1}\right)$ given by $\mathcal{P}=\left\{ F_{1}>0\right\}$  and $\mathcal{N}=\left\{ F_{1}=0\right\}$, for some real-valued non-negative $\widehat{\f a}^{\left(1\right)}$-invariant function $F_{1}\in L^{1}\left(\mu_{1}\right)$. 

We then claim that the decomposition $\Omega=\mathcal{P}\cup\mathcal{N}$ serves also as the positive-null decomposition with respect to $\mu$. First, obviously the existence or non-existence of a.c.i.p. on an $\f a$-invariant set is not affected by passing to a mutually absolutely continuous measure. Second, to find the special function in $L^{1}\left(\mu\right)$ representing this decomposition with respect to $\mu$, we simply choose
$$F:=\mathfrak{D}\left(F_{1}\right)\in L^{1}\left(\mu\right).$$
Obviously $\mathcal{P}=\left\{ F>0\right\}$  and $\mathcal{N}=\left\{ F=0\right\}$. Also, $F$ is $\widehat{\f a}$-invariant since, using the identity \eqref{eq: DaaD} and the $\widehat{\f a}^{\left(1\right)}$-invariance of $F_{1}$, for every $g\in G$,
$$\widehat{\f a}_{g}F=\mathfrak{D}\left(\widehat{\f a}_{g}^{\left(1\right)}F_{1}\right)=\mathfrak{D}\left(F_{1}\right)=F.$$
Finally, given a Følner sequence $\left(F_{N}\right)_{N=1}^{\infty}$ for $G$, for every $f\in L^{1}\left(\mu\right)$ we can use the identity \eqref{eq: DaaD} to write, for every $N$,
$${\textstyle \frac{1}{\left|F_{N}\right|} \sum_{g\in F_{N}}}\widehat{\f a}_{g}f=\mathfrak{D}\left({\textstyle \frac{1}{\left|F_{N}\right|}\sum_{g\in F_{N}}}\widehat{\f a}_{g}\mathfrak{D}^{-1}\left(f\right)\right).$$
By the properties of $\mathfrak{D}$, since the last part of Theorem \ref{Theorem: Positive-Null Decomposition} holds for $\mu_{1}$ it holds also for $\mu$.
\end{proof}

From the positive-null decomposition for $G$-systems we obtain a positive-null decomposition for $S\alpha S$ $G$-processes using the spectral representation in the standard way described by Samorodnitsky \cite{samorodnitsky2005}. Suppose that $\mathbf{X}$ is a stationary non-Gaussian $S\alpha S$ $G$-process with minimal spectral representation
$$\mathbf{X}\ed\left(\int_{\Omega}\f u_{g}f_{0}dM\right)_{g\in G}.$$
Applying the positive-null decomposition to the underlying $G$-system $G\overset{\f a}{\curvearrowright}\left(\Omega,\mu\right)$ we obtain corresponding positive and null parts $\mathcal{P}$ and $\mathcal{N}$. We then define the positive-null decomposition of $\mathbf{X}$ into $\mathbf{X}\ed\mathbf{X}^{\mathcal{P}}+\mathbf{X}^{\mathcal{N}}$ by letting
$$\mathbf{X}^{\mathcal{P}}=\left(\int_{\mathcal{P}}\f u_g f_{0}dM\right)_{g\in G}\,\,\,\text{ and  }\,\,\,\mathbf{X}^{\mathcal{N}}=\left(\int_{\mathcal{N}}\f u_g f_{0}dM\right)_{g\in G}.$$
Recalling the structure of Rosiński minimal spectral representation as in Theorem \ref{Theorem: Rosiński} and using the $\f a$-invariance of $\mathcal{P}$ and $\mathcal{N}$, we see that $\mathbf{X}^{\mathcal{P}}$ and $\mathbf{X}^{\mathcal{N}}$ are independent stationary non-Gaussian $S\alpha S$ $G$-processes. Thus, $\mathbf{X}$ has null (positive, resp.) spectral representation exactly when $\mathbf{X}\ed\mathbf{X}^{\mathcal{N}}$ ($\mathbf{X}\ed\mathbf{X}^{\mathcal{P}}$, resp.).

\section{Ergodic properties of $S\alpha S$ processes via the spectral representation}

\subsection{Characterization of weak-mixing}

In his work \cite{gross}, Gross gave an important condition for strong-mixing in stationary non-Gaussian $S\alpha S$ processes in terms of the spectral representation of the process. Here we introduce the translation of this condition for weak-mixing, using the notion of density for subsets of the group.

Let $G$ be a countable amenable group and fix a Følner sequence $\mathcal{F}=\left(F_{N}\right)_{N=1}^{\infty}$ for $G$. For a subset $D\subset G$ define the {\it $\mathcal{F}$-density} of $D$ to be
$$\mathrm{d}_{\mathcal{F}}\left(D\right)=\lim_{N\to\infty}\frac{\left|D\cap F_{N}\right|}{\left|F_{N}\right|}=\lim_{N\to\infty}\frac{1}{\left|F_{N}\right|}\sum_{g\in F_{N}}1_{D}\left(g\right)$$
when the limit exists. When $\mathrm{d}_{\mathcal{F}}\left(D\right)=1$ we say that $D$ has {\it full $\mathcal{F}$-density}. The Følner property of $\mathcal{F}$ makes the assignment $D\mapsto\mathrm{d}_{\mathcal{F}}\left(D\right)$ translation-invariant when it is well-defined, but in the following discussion we will not use this property. 

Recall the notation for limits as in \eqref{eq: countable limits}. By the classical Koopman--von Neumann Lemma (see \cite[Lemma~6.2 in \textsection 2.6]{petersen}), a real-valued non-negative function $\psi$ on $G$ satisfies
$$\lim_{N\to\infty}\frac{1}{\left|F_{N}\right|}\sum_{g\in F_{N}}\psi\left(g\right)=0$$
if and only if there exists a full $\mathcal{F}$-density set $D\subset G$ such that
$$\lim_{g\in D}\psi\left(g\right)=0.$$

The following theorem is the weak-mixing version of Gross' condition for strong-mixing as in \cite[Proposition~4.2]{wangroystoev}.

\begin{theorem}[Wang, Roy \& Stoev, following Gross]
Let $G$ be a countable amenable group and $\mathbf{X}$ be a stationary non-Gaussian $S\alpha S$ $G$-process with Rosiński minimal spectral representation as in Theorem \ref{Theorem: Rosiński}. Then $\mathbf{X}$ is weakly-mixing if and only if for every Følner sequence $\mathcal{F}$ for $G$ there exists a set $D\subset G$ of full $\mathcal{F}$-density such that, for every compact set $K\subset\left(0,\infty\right)$ and every $\epsilon>0$,
$$\lim_{g\in D}\mu\left(\left|f_{0}\right|^{\alpha}\in K,\left|f_{g}\right|^{\alpha}>\epsilon\right)=0.$$
\end{theorem}

This formulation was used in the proof of \cite[Theorem~4.1]{wangroystoev} as well as in the proof of \cite[Theorem~3.1]{samorodnitsky2005}. In the following we present a slight strengthening of the Koopman--von Neumann Lemma in order to show that, in the above theorem, it is sufficient to find a full $\mathcal{F}$-density set $D=D\left(K,\epsilon\right)$ that may depend on the choice of $K$ and $\epsilon$.

\begin{lemma}[A Simultaneous Koopman--von Neumann Lemma]
\label{Lemma: Simultaneous Kopmman-von Neumann}
Let $G$ be a countable amenable group with a Følner sequence $\mathcal{F}=\left(F_{N}\right)_{N=1}^{\infty}$. For a countable collection of real-valued non-negative functions $\Psi\subset l^{\infty}\left(G\right)$ the following are equivalent.
\begin{enumerate}
\item $\lim_{N\to\infty}\frac{1}{\left|F_{N}\right|}\sum_{g\in F_{N}}\psi\left(g\right)=0$ for all $\psi\in\Psi$.
\item There is $D\subset G$ of full $\mathcal{F}$-density such that $\lim_{g\in D}\psi\left(g\right)=0$ for all $\psi\in\Psi$.
\end{enumerate}
\end{lemma}

\begin{proof}
The implication $\left(2\right)\implies\left(1\right)$ is an easy consequence of the Cesàro summation.  We show that $\left(1\right)\implies\left(2\right)$. Pick an ordering $\Psi=\left(\psi^{\left(k\right)}\right)_{k=1}^{\infty}$ and, for every $k$, consider the sets
$$E_{m}^{\left(k\right)}=\left\{ g\in G:\psi^{\left(k\right)}\left(g\right)>1/m\right\} ,\quad m=1,2,\dotsc.$$
By the assumption we see that $\ensuremath{\mathrm{d}_{\mathcal{F}}\left(E_{m}^{\left(k\right)}\right)=0}$ for every $k$ and $m$. Consider the sets
$$E_{m}=E_{m}^{\left(1\right)}\cup E_{m}^{\left(2\right)}\cup\dotsm\cup E_{m}^{\left(m\right)},\quad m=1,2,\dotsc.$$
Since each $E_{m}$ is a finite union of sets of $\mathcal{F}$-density zero, $\mathrm{d}_{\mathcal{F}}\left(E_{m}\right)=0$ for all $m$. Hence we can fix an increasing sequence $\left(N_{m}\right)_{m=2}^{\infty}$ of the positive integers such that, for every $m$,
$${\textstyle N\geq N_{m-1}\implies\frac{1}{\left|F_{N}\right|}\sum_{g\in F_{N}}1_{E_{m}}\left(g\right)<\frac{1}{m}}.$$
Then put
$$E={\textstyle \bigcup_{m=2}^{\infty}}\left(E_{m}\cap\left(F_{N_{m}}\backslash F_{N_{m-1}}\right)\right)\text{ and }D=G\backslash E.$$
Now we can proceed in a complete analogy with the proof of \cite[Lemma~6.2 in \textsection 2.6]{petersen} to prove that $\mathrm{d}_{\mathcal{F}}\left(D\right)=1$ and that $\lim_{g\in D}\psi^{\left(k\right)}\left(g\right)=0$ for each $k$.
\end{proof}

In the following theorem we introduce some conditions that are equivalent to weak-mixing. Condition $\left(2\right)$ is due to Cambanis, Hardin \& Weron \cite[Theorem~2]{CaHaWe}; Condition $\left(3\right)$ is the weak-mixing version of Gross' \cite[Theorem~2.7]{gross} as used by Samorodnitsky \cite{samorodnitsky2005} and Wang, Roy \& Stoev \cite[Theorem~4.2]{wangroystoev}; Condition $\left(4\right)$ is new and is a result of the above simultaneous Koopman--von Neumann Lemma. It is worthwhile to mention that the minimality of the Rosiński spectral representation does not play any role here (cf. \cite[\textsection 4]{gross}).

\begin{theorem}
\label{Theorem: Weak-mixing Characterization}
Let $G$ be a countable amenable group and $\mathbf{X}$ a stationary non-Gaussian $S\alpha S$ $G$-process with Rosiński (possibly non-minimal) spectral representation as in Theorem \ref{Theorem: Rosiński}. Denote by $\mathrm{Lin}\left(f_{0}\right)\subset L^{\alpha}\left(\mu\right)$ the linear subspace spanned by $\left\{ \f u_{g}f_{0}:g\in G\right\}$. Fix a Følner sequence $\mathcal{F}=\left(F_N\right)_{N=1}^{\infty}$ for $G$. The following are equivalent.
\begin{enumerate}
\item $\mathbf{X}$ is weakly-mixing;
\item for every $f_{1},f_{2}\in\mathrm{Lin}\left(f_{0}\right)$,
$$\lim_{N\to\infty}\frac{1}{\left|F_{N}\right|}\sum_{g\in F_{N}}\left|\exp\left(-\no{f_{1}+\f u_{g}f_{2}}_{\alpha}^{\alpha}\right)-\exp\left(-\no{f_{1}}_{\alpha}^{\alpha}\right)\exp\left(-\no{f_{2}}_{\alpha}^{\alpha}\right)\right|=0;$$
\item for every $f_{1},f_{2}\in\mathrm{Lin}\left(f_{0}\right)$ there exists a set $D\subset G$ of full $\mathcal{F}$-density such that, for every compact set $K\subset\left(0,\infty\right)$ and every $\epsilon>0$,
$$\lim_{g\in D}\mu\left(\left|f_{1}\right|\in K,\left|\f u_{g}f_{2}\right|>\epsilon\right)=0;$$
\item for every $f_{1},f_{2}\in\mathrm{Lin}\left(f_{0}\right)$, every compact set $K\subset\left(0,\infty\right)$ and every $\epsilon>0$, there exists a set $D\left(K,\epsilon\right)\subset G$ of full $\mathcal{F}$-density such that
$$\lim_{g\in D\left(K,\epsilon\right)}\mu\left(\left|f_{1}\right|\in K,\left|\f u_{g}f_{2}\right|>\epsilon\right)=0.$$
\end{enumerate}
Moreover, each of Conditions $\left(3\right)$ and $\left(4\right)$ is enough to be verified for $f_{1}=f_{0}$ and $f_{2}=\f u_{g}f_{0}$ for each arbitrary $g\in G$.
\end{theorem}

For the reader convenience we refer to the proof of each of the above implications. The only new argument is in the implication of $\left(4\right)\implies\left(1\right)$.

\begin{itemize}
\item The equivalence of $\left(1\right)$ and $\left(2\right)$ was established by Cambanis, Hardin \& Weron in \cite[Theorem~2]{CaHaWe} in their characterization of strong-mixing for $G=\mathbb{Z}$ and, passing to a set of full $\mathcal{F}$-density and using the Koopman--von Neumann Lemma, the same proof remains valid for the weak-mixing of a process indexed by a countable amenable group $G$.
\item The implication $\left(2\right)\implies\left(3\right)$ was established by Gross in \cite[Proposition~2.5(i)]{gross} together with the observation at the beginning of the proof of \cite[Theorem~2.7]{gross} in his characterization of strong-mixing for $G=\mathbb{Z}$. Passing to a set of full $\mathcal{F}$-density and using the Koopman--von Neumann Lemma, the same proof remains valid for the weak-mixing of a process indexed by a countable amenable group $G$. We remark that it is visible from the proof of \cite[Proposition~2.5(i)]{gross} that the assumption of $\left(2\right)$ provides a set $D$ of full $\mathcal{F}$-density that works simultaneously for all $K$ and $\epsilon$.
\item The implication $\left(3\right)\implies\left(4\right)$ is trivial.
\item The implication $\left(4\right)\implies\left(1\right)$ can be seen as follows. Let $\mathcal{C}$ be the collection of pairs $\left(K,\epsilon\right)$ where $K$ is a compact interval with positive rational endpoints and $\epsilon$ is the reciprocal of an arbitrary positive integer. Then $\left(4\right)$ is true (in particular) for all $\left(K,\epsilon\right)\in\mathcal{C}$ and $\mathcal{C}$ is countable, so that from the above Lemma \ref{Lemma: Simultaneous Kopmman-von Neumann} there is a set $D\subset G$ of full $\mathcal{F}$-density that satisfies $\left(3\right)$ simultaneously for all $\left(K,\epsilon\right)\in\mathcal{C}$. Finally, from Gross' proof of $\left(3\right)\implies\left(1\right)$ as in \cite[Theorem~2.7]{gross} it is clear that it is enough to assume $\left(3\right)$ only for $\left(K,\epsilon\right)\in\mathcal{C}$ in order to conclude $\left(1\right)$, so we conclude that $\left(4\right)\implies\left(1\right)$.
\item Finally, the fact that $\left(3\right)$ is sufficient to be verified for $f_{1}=f_{0}$ and $f_{2}=\f u_{g}f_{0}$ for each arbitrary $g\in G$ in order to deduce weak-mixing, was mentioned by Gross in \cite[Theorem~2.1]{gross} following Maruyama \cite{maruyama}. It was justified in \cite[Appendix~A]{wangroystoev} and the same reasoning is applied in showing the analogous statement for $\left(4\right)$.
\end{itemize}

We can now prove Theorems \ref{Theorem: weak-mixing null} and \ref{Theorem: Main Theorem}. In the following Propositions \ref{Proposition: ErgodicThenNull} and \ref{Proposition: NullThenErgodic}, $\mathbf{X}$ stands for a stationary non-Gaussian $S\alpha S$ $G$-process with Rosiński spectral representation as in Theorem \ref{Theorem: Rosiński}. In Proposition \ref{Proposition: ErgodicThenNull} we establish Theorem \ref{Theorem: weak-mixing null} and one implication in Theorem \ref{Theorem: Main Theorem} without assuming minimality of the Rosiński spectral representation. In Proposition \ref{Proposition: ErgodicThenNull} we establish the converse implication in Theorem \ref{Theorem: Main Theorem} and for that we do assume that the Rosiński spectral representation is minimal.

\begin{proposition}
\label{Proposition: NullThenErgodic}
If $\mathbf{X}$ has a null Rosiński (possibly non-minimal) spectral representation then it is weakly-mixing.
\end{proposition}

\begin{proof}
We aim to show Condition $\left(4\right)$ in Theorem \ref{Theorem: Weak-mixing Characterization} in order to establish that $\mathbf{X}$ is weakly-mixing. Let $K\subset\left(0,\infty\right)$ be a compact set and $\epsilon>0$. Put $E=\left\{ \left|f_{0}\right|^{\alpha}\in K\right\}$. Then $\mu\left(E\right)<\infty$ since $\inf K>0$ and $f_{0}\in L^{\alpha}\left(\mu\right)$. If $\mu\left(E\right)=0$ there is nothing to prove so we assume also that $\mu\left(E\right)>0$. Let $\left(F_{N}\right)_{N=1}^{\infty}$ be a Følner sequence for $G$ and, recalling the dual action defined in \ref{eq: dual action}, we put the random variables
$$W_{N}:={\textstyle \frac{1}{\left|F_{N}\right|}\sum_{g\in F_{N}}}\left|f_{g}\right|^{\alpha}={\textstyle \frac{1}{\left|F_{N}\right|}\sum_{g\in F_{N}}}\widehat{\f a}_{g}\left|f_{0}\right|^{\alpha},\quad N=1,2,\dotsc.$$
Let $\eta>0$ be arbitrary and write
\begin{align*}
& {\textstyle \frac{1}{\left|F_{N}\right|}\sum_{g\in F_{N}}}\mu\left(E\cap\left\{ \left|f_{g}\right|^{\alpha}>\epsilon\right\} \right) = \int_{E}{\textstyle \frac{1}{\left|F_{N}\right|}\sum_{g\in F_{N}}}1_{\left\{ \left|f_{g}\right|^{\alpha}>\epsilon\right\} }d\mu \\
& \qquad = \int_{E\cap\left\{ W_{N}\leq\eta\right\} }\dotsm d\mu+\int_{E\cap\left\{ W_{N}>\eta\right\} }\dotsm d\mu:=A_{N}\left(\eta\right)+B_{N}\left(\eta\right).
\end{align*}
For the term $A_{N}\left(\eta\right)$, using that $1_{\left\{ \left|f_{g}\right|^{\alpha}>\epsilon\right\} }\leq\epsilon^{-1}\left|f_{g}\right|^{\alpha}$ we have that
$$A_{N}\left(\eta\right)\leq\epsilon^{-1}\int_{E\cap\left\{ W_{N}\leq\eta\right\} }W_{N}d\mu\leq\epsilon^{-1}\eta\mu\left(E\right).$$
As for the term $B_{N}\left(\eta\right)$, bounding the integrand by $1$ we have
$$B_{N}\left(\eta\right)\leq\mu\left(E\cap\left\{ W_{N}>\eta\right\} \right)=\mu\left(E\right)\mu_{E}\left(W_{N}>\eta\right),$$
where $\mu_{E}\left(\cdot\right)$ denotes the probability measure obtained from $\mu$ by conditioning on the finite measure set $E$. Obviously, $\mu_{E}$ is absolutely continuous with respect to $\mu$. Hence, if there is an a.c.i.p. with respect to $\mu_{E}$ then it is also an a.c.i.p. with respect to $\mu$, contradicting the assumption that the given spectral representation of $\mathbf{X}$ is null. Thus, there is no a.c.i.p. for $\mu_{E}$ and, from Theorem \ref{Theorem: Krengel Amenable} applied to $\left|f_{0}\right|^{\alpha}\in L^{1}\left(\mu_{E}\right)$, we see that
$$\lim_{N\to\infty}B_{N}\left(\eta\right)\leq\mu\left(E\right)\lim_{N\to\infty}\mu_{E}\left(W_{N}>\eta\right)=0.$$
As this holds for an arbitrary positive $\eta$ we obtain that
$$\lim_{N\to\infty}{\textstyle \frac{1}{\left|F_{N}\right|}\sum_{g\in F_{N}}}\mu\left(\left|f_{0}\right|^{\alpha}\in K,\left|f_{g}\right|^{\alpha}>\epsilon\right)=\lim_{N\to\infty}{\textstyle \frac{1}{\left|F_{N}\right|}\sum_{g\in F_{N}}}\mu\left(E\cap\left\{ \left|f_{g}\right|^{\alpha}>\epsilon\right\} \right)=0.$$
Using the Koopman--von Neumann Lemma we see that Condition $\left(4\right)$ in Theorem \ref{Theorem: Weak-mixing Characterization} holds for $f_1=f_0$ and $f_2=f_g=\f u_g f_0$ with an arbitrary $g\in G$. In view of the last part of Theorem \ref{Theorem: Weak-mixing Characterization} we see that $\mathbf{X}$ is weakly-mixing.
\end{proof}

\begin{proposition}
\label{Proposition: ErgodicThenNull}
If $\mathbf{X}\neq 0$ and is ergodic then its Rosiński minimal spectral representation is null.
\end{proposition}

Before we prove Proposition \ref{Proposition: ErgodicThenNull} let us recall the following functoriality of the Rosiński spectral representation. Suppose that $\mathbf{X}$ has a Rosiński spectral representation over the $G$-system $G\overset{\f a}{\curvearrowright}\left(\Omega,\mu\right)$. We then write
$$\mathbf{X}\ed\left(\int_{\Omega}\f u_{g}^{\left(\mu\right)}f_{0}dM_{\alpha}^{\left(\mu\right)}\right)_{g\in G},$$
where $M_{\alpha}^{\left(\mu\right)}$ is an $S\alpha S$ random measure controlled by $\mu$, $f_{0}\in L^{\alpha}\left(\mu\right)$ and
$$\f u_{g}^{\left(\mu\right)}f=c_{g}\sqrt[\alpha]{\frac{d\mu\circ\f a_{g}}{d\mu}}f_{0}\circ\f a_{g},\quad g\in G.$$
If $\nu$ is any probability measure that is mutually absolutely continuous with $\mu$, we can produce another Rosiński spectral representation of $\mathbf{X}$ over $G\overset{\f a}{\curvearrowright}\left(\Omega,\nu\right)$ as follows. Put
$$\Delta:=\frac{d\nu}{d\mu}\in L^{1}\left(\mu\right)$$
and note the relation
$$\frac{d\mu\circ\f a_{g}}{d\mu}=\frac{\Delta}{\Delta\circ\f a_{g}}\frac{d\nu\circ\f a_{g}}{d\nu},\quad g\in G.$$
Letting $\f u_{g}^{\left(\nu\right)}$ be defined analogously to $\f u_{g}^{\left(\mu\right)}$, we see that for every $g\in G$,
$$\f u_{g}^{\left(\mu\right)}f_{0}=c_{g}\sqrt[\alpha]{\frac{\Delta}{\Delta\circ\f a_{g}}}\sqrt[\alpha]{\frac{d\nu\circ\f a_{g}}{d\nu}}f_{0}\circ\f a_{g}=\sqrt[\alpha]{\Delta}\f u_{g}^{\left(\nu\right)}\left(f_{0}/\sqrt[\alpha]{\Delta}\right).$$
It is then routine to check that
$$\mathbf{X}\ed\left(\int_{\Omega}\f u_{g}^{\left(\nu\right)}\left(f_{0}/\sqrt[\alpha]{\Delta}\right)dM_{\alpha}^{\left(\nu\right)}\right)_{g\in G}$$
is also a Rosiński spectral representation of $\mathbf{X}$, where $M_{\alpha}^{\left(\nu\right)}$ is an $S\alpha S$ random measure controlled by $\nu$.

\begin{proof}[Proof of Proposition \ref{Proposition: ErgodicThenNull}]
By Theorem \ref{Theorem: Ergodicity iff Weak-Mixing} we may assume that $\mathbf{X}$ is weakly-mixing. Let us first show that the null part of the Rosiński spectral representation of $\mathbf{X}$ does not vanish.  Suppose toward a contradiction otherwise, so that by Theorem \ref{Theorem: Positive-Null Decomposition} there exists an a.c.i.p. $\nu$ which is mutually absolutely continuous with $\mu$ that is underlying the Rosiński minimal spectral representation.  From the above discussion we may assume that the Rosiński spectral representation is over $G\overset{\f a}{\curvearrowright}\left(\Omega,\nu\right)$, where $\nu$ is an $\f a$-invariant probability measure. Let $f_0\in L^{\alpha}\left(\nu\right)$ be the special function underlying this Rosiński spectral representation.

We now show that we are in violation of Condition $\left(4\right)$ for weak-mixing as in Theorem \ref{Theorem: Weak-mixing Characterization}.  Note that by the $\f a$-invariance of $\nu$, simply
$$\f u_{g}^{\left(\nu\right)}f=f\circ\f a_{g},\quad g\in G,\quad f\in L^1\left(\nu\right).$$
Let $\ip{\cdot}{\cdot}_{\nu}$ denotes the standard inner product on $L^{2}\left(\nu\right)$ and $\mathrm{Inv}\left(\f a\right)$ the σ-algebra of $\f a$-invariant Borel sets. From the von Neumann Ergodic Theorem, for every compact set $K\subset\left(0,\infty\right)$,
\begin{align*}
    & {\textstyle \frac{1}{\left|F_{N}\right|}\sum_{g\in F_{N}}}\nu\left(\left|f_{0}\right|^{\alpha}\in K,\left|\f u_{g}^{\left(\nu\right)}\left(f_{0}\right)\right|^{\alpha}\in K\right) \\
    & \qquad = \ip{1_{\left\{ \left|f_{0}\right|^{\alpha}\in K\right\} }}{{\textstyle \frac{1}{\left|F_{N}\right|}\sum_{g\in F_{N}}}1_{\left\{ \left|f_{0}\right|^{\alpha}\in K\right\} }\circ\f a_{g}}_{\nu} \\
    & \qquad\qquad \xrightarrow[N\to\infty]{} \ip{1_{\left\{ \left|f_{0}\right|^{\alpha}\in K\right\} }}{\mathbb{E}\left(1_{\left\{ \left|f_{0}\right|^{\alpha}\in K\right\} }\mid\mathrm{Inv}\left(\f a\right)\right)}_{\nu}.
\end{align*}
Since $\nu$ is a probability measure, $1_{\left\{ \left|f_{0}\right|^{\alpha}\in K\right\} }\nearrow1$ as $K\nearrow\left(0,\infty\right)$ in $L^{2}\left(\nu\right)$. Then for a sufficiently large $K$ the last term is strictly positive. Then, in view of Theorem \ref{Theorem: Weak-mixing Characterization}, it follows that $\mathbf{X}$ is not weakly-mixing.

We have shown that, as $\mathbf{X}$ is weakly-mixing, the null part in every Rosiński spectral representation of $\mathbf{X}$ does not vanish. Next we show further that when the Rosiński spectral representation of $\mathbf{X}$ is minimal its positive part $\mathcal{P}$ vanishes. Suppose toward a contradiction that $\mathcal{P}$ is of positive measure. By the minimality it follows that $\mathbf{X}^{\mathcal{P}}\neq 0$ and by the previous part of the proof $\mathbf{X}^{\mathcal{P}}$ is not weakly-mixing hence, by Theorem \ref{Theorem: Ergodicity iff Weak-Mixing}, it is also non-ergodic. Then following the very same argument as in the end of Samorodnitsky's proof of \cite[Theorem~3.1]{samorodnitsky2005} we conclude that $\mathbf{X}$ itself is non-ergodic, which is a contradiction.
\end{proof}

\subsection{Characterization of positive and null Rosiński minimal spectral representation}

We now formulate tests for processes having positive or null Rosiński minimal spectral representation. We recall that a Følner sequence $\left(F_N\right)_{N=1}^{\infty}$ is called {\it tempered} if there exists a positive constant $C$ such that, for every $N$,
$$\left|{\textstyle \bigcup_{K\leq N}}F_{K}^{-1}F_{N+1}\right|\leq C\left|F_{N+1}\right|.$$
By \cite[Proposition~1.4]{lindenstrauss} every Følner sequence has a tempered Følner subsequence.  By the Lindenstrauss Pointwise Ergodic Theorem \cite[Theorem~1.2]{lindenstrauss}, if $G\overset{\f a}{\curvearrowright}\left(\Omega,\mu\right)$ is a measure preserving action of a countable amenable group $G$ on a standard probability space $\left(\Omega,\mu\right)$, then for every tempered Følner sequence $\left(F_{N}\right)_{N=1}^{\infty}$ for $G$ and every $f\in L^{1}\left(\mu\right)$,
$$\lim_{N\to\infty}{\textstyle \frac{1}{\left|F_{N}\right|}\sum_{g\in F_{N}}}f\circ\f a_{g}=\mathbb{E}\left(f\mid\mathrm{Inv}\left(\f a\right)\right)\text{ almost surely.}$$
We shall exploit the fact that the limit in the Lindenstrauss Pointwise Ergodic Theorem holds, in particular, also in measure.

\begin{theorem}
\label{Theorem: Positive Null}
Let $G$ be a countable amenable group and $\mathbf{X}$ a stationary non-Gaussian $S\alpha S$ $G$-process with Rosiński minimal spectral representation as in Theorem \ref{Theorem: Rosiński}.
\begin{enumerate}
\item This spectral representation of $\mathbf{X}$ is null if and only if for every (equivalently, there exists a) Følner sequence $\left(F_{N}\right)_{N=1}^{\infty}$ for $G$ we have that
$$\frac{1}{\left|F_{N}\right|}\sum_{g\in F_{N}}\widehat{\f a}_{g}\left|f_{0}\right|^{\alpha}\xrightarrow[N\to\infty]{}0\text{ in measure}.$$
\item This spectral representation of $\mathbf{X}$ is positive if and only if every (equivalently, there exists a) Følner sequence $\left(F_{N}\right)_{N=1}^{\infty}$ for $G$ admits a (tempered) Følner subsequence $\left(F_{N_{M}}\right)_{M=1}^{\infty}$ such that
$$\frac{1}{\left|F_{N_{M}}\right|}\sum_{g\in F_{N_{M}}}\widehat{\f a}_{g}\left|f_{0}\right|^{\alpha}\xrightarrow[M\to\infty]{}\mathbb{E}\left(\left|f_{0}\right|^{\alpha}\mid\mathrm{Inv}\left(\f a\right)\right)\text{ almost surely}.$$
\end{enumerate}
\end{theorem}

Note that since the conditional expectation is a strictly positive operator and the Rosiński spectral representation is assumed to be minimal, $\mathbb{E}\left(\left|f_{0}\right|^{\alpha}\mid\mathrm{Inv}\left(\f a\right)\right)$ is positive almost surely. Thus, the two alternatives in Theorem \ref{Theorem: Positive Null} stand in the opposite extremes of each other.

For the ease of notation in the following proof, for $F\Subset G$ we let $A_{F}:L^{1}\left(\mu\right)\to L^{1}\left(\mu\right)$ be the operator
$$A_{F}f=\frac{1}{\left|F\right|}\sum_{g\in F}\widehat{\f a}_{g}f.$$
Note that if $E\in\mathrm{Inv}\left(\f a\right)$ then $A_{F}\left(1_{E}f\right)=1_{E}A_{F}\left(f\right)$. 

\begin{proof}[Proof of Theorem \ref{Theorem: Positive Null}]
We prove the first part and the second part can be proved in a similar fashion. If the Rosiński minimal spectral representation of $\mathbf{X}$ is null, then by Theorem \ref{Theorem: Krengel Amenable} applied to $\left|f_{0}\right|^{\alpha}\in L^{1}\left(\mu\right)$ we have $A_{F_{N}}\left|f_{0}\right|^{\alpha}\xrightarrow[N\to\infty]{\mu}0$ for every Følner sequence $\left(F_{N}\right)_{N=1}^{\infty}$ for $G$.

Conversely, suppose that $A_{F_{N}}\left|f_{0}\right|^{\alpha}\xrightarrow[N\to\infty]{\mu}0$ for some Følner sequence $\left(F_{N}\right)_{N=1}^{\infty}$ for $G$. By the definition of the positive part $\mathcal{P}$ in Theorem \ref{Theorem: Positive-Null Decomposition} the measure $\mu\mid_{\mathcal{P}}$ has an a.c.i.p. that we denote $\nu$. Consider the function
$$1_{\mathcal{P}}\frac{d\mu}{d\nu}\left|f_{0}\right|^{\alpha}\in L^{1}\left(\nu\right).$$
Applying the Lindenstrauss Pointwise Ergodic Theorem, there is a subsequence $\left(N_{M}\right)_{M=1}^{\infty}$ of the positive integers such that $\left(F_{N_M}\right)_{M=1}^{\infty}$ is tempered and
$$\lim_{M\to\infty}{\textstyle \frac{1}{\left|F_{N_M}\right|}\sum_{g\in F_{N_M}}}\left(1_{\mathcal{P}}\frac{d\mu}{d\nu}\left|f_{0}\right|^{\alpha}\right)\circ\f a_{g}=\mathbb{E}\left(1_{\mathcal{P}}\frac{d\mu}{d\nu}\left|f_{0}\right|^{\alpha}\mid\mathrm{Inv}\left(\f a\right)\right)\text{ \ensuremath{\nu}-almost surely.}$$
However, since $\mathcal{P}$ is an $\f a$-invariant set and $\nu$ is an $\f a$-invariant measure, for every $g\in G$ we may write
$$\left(1_{\mathcal{P}}\frac{d\mu}{d\nu}\left|f_{0}\right|^{\alpha}\right)\circ\f a_{g}=1_{\mathcal{P}}\frac{d\mu\circ\f a_{g}}{d\nu}\left|f_{0}\right|^{\alpha}\circ\f a_{g}=1_{\mathcal{P}}\frac{d\mu}{d\nu}\widehat{\f a}_{g}\left|f_{0}\right|^{\alpha}.$$
It follows that
$$1_{\mathcal{P}}\frac{d\mu}{d\nu}\lim_{M\to\infty}A_{F_{N_{M}}}\left|f_{0}\right|^{\alpha}=\mathbb{E}\left(1_{\mathcal{P}}\frac{d\mu}{d\nu}\left|f_{0}\right|^{\alpha}\mid\mathrm{Inv}\left(\f a\right)\right)\text{ \ensuremath{\nu\text{- hence }\mu}-almost surely.}$$
Here we used that on $\mathcal{P}$, $\nu$ and $\mu$ are mutually absolutely continuous. Together with the assumption $A_{F_{N}}\left|f_{0}\right|^{\alpha}\xrightarrow[N\to\infty]{\mu}0$ we obtain that
$$\mathbb{E}\left(1_{\mathcal{P}}\frac{d\mu}{d\nu}\left|f_{0}\right|^{\alpha}\mid\mathrm{Inv}\left(\f a\right)\right)=0\text{ }\mu\text{-almost surely.}$$
Since the conditional expectation is a strictly positive operator and $d\mu/d\nu$ is positive almost surely, we conclude that $1_{\mathcal{P}}\left|f_{0}\right|^{\alpha}=0$, showing that the Rosiński minimal spectral representation of $\mathbf{X}$ is null.
\end{proof}

\section{Constructions and examples}

Here we use constructions from ergodic theory to produce stationary non-Gaussian $S\alpha S$ processes on amenable groups with interesting ergodic properties.

Let $G$ be a countable group. A $G$-system $G\overset{\f a}{\curvearrowright}\left(\Omega,\mu\right)$ is called {\it conservative} if it has no wandering set of positive measure, where a Borel set $W$ in $\Omega$ is {\it wandering} if the sets in $\left\{\f a_g\left(W\right):g\in G\right\}$ are pairwise disjoint up to a $\mu$-null set. On the other extreme, a $G$-system $G\overset{\f a}{\curvearrowright}\left(\Omega,\mu\right)$ is called {\it totally dissipative} if there exists a wandering set $W$ in $\Omega$ such that, up to a $\mu$-null set,
$$\Omega={\textstyle \bigcup_{g\in G}}\f a_g\left(W\right).$$
By the well-known {\it Hopf Decomposition}, for every $G$-system $G\overset{\f a}{\curvearrowright}\left(\Omega,\mu\right)$ there is a partition
$$\Omega=\mathcal{C}\cup\mathcal{D}$$
of $\Omega$ into disjoint $\f a$-invariant sets $\mathcal{C}$ and $\mathcal{D}$ such that the $G$-system $G\overset{\f a}{\curvearrowright}\left(\mathcal{C},\mu\right)$ is conservative and the $G$-system $G\overset{\f a}{\curvearrowright}\left(\mathcal{D},\mu\right)$ is totally dissipative. For detailed presentations we refer to \cite[\textsection 3.1]{krengel1985} and \cite[\textsection 1.6]{aaronson}.

Suppose that $\mathbf{X}$ is a stationary non-Gaussian $S\alpha S$ $G$-process with minimal spectral representation as in Definition \ref{Definition: Spectral Representation},
\begin{equation}
\mathbf{X}\ed\left(\int_{\Omega}\f u_{g}f_{0}dM_{\alpha}\right)_{g\in G}.
\end{equation}
Let $\mathcal{C}\cup\mathcal{D}$ be the Hopf Decomposition of the $G$-system underlying this spectral representation. As $\mathcal{C}$ and $\mathcal{D}$ are invariant sets, we can write $\mathbf{X}\ed\mathbf{X}^{\mathcal{C}}+\mathbf{X}^{\mathcal{D}}$ where
$$\mathbf{X}^{\mathcal{C}}=\left(\int_{\mathcal{C}}\f u_g f_{0}dM_{\alpha}\right)_{g\in G}\,\,\,\text{ and  }\,\,\,\mathbf{X}^{\mathcal{D}}=\left(\int_{\mathcal{D}}\f u_g f_{0}dM_{\alpha}\right)_{g\in G}.$$
Then $\mathbf{X}^{\mathcal{C}}$ and $\mathbf{X}^{\mathcal{D}}$ are independent stationary non-Gaussian $S\alpha S$ $G$-processes (cf. \cite{rosinski1995, roysam}).

It turns out that the process $\mathbf{X}^{\mathcal{D}}$ has always a certain structure called {\it mixed moving average} \cite[Corollary~4.6]{rosinski1995}. Many properties of mixed moving average $S\alpha S$ processes are known, among them is that they are strongly-mixing \cite[Theorem~3]{surgailis1993}. Thus, in the constructions to follow we focus on stationary $S\alpha S$ $G$-processes arising from conservative $G$-systems.

\begin{remark}
\label{Remark: conservative part}
The Hopf Decomposition of a stationary $S\alpha S$ $G$-process is defined with respect to a {\it minimal} spectral representation. Yet, if $\mathbf{X}$ is a stationary $S\alpha S$ $G$-process with an arbitrary conservative spectral representation then $\mathbf{X}^{\mathcal{D}}$ vanishes. In order to see this, observe that by the virtue of Hardin's proof of the existence of minimal spectral representation \cite[Theorem~5.1]{hardin}, if we are given any spectral representation of $\mathbf{X}$, its minimal spectral representation can be realized by passing to a sub-σ-algebra of the $G$-system underlying the given spectral representation. As conservativeness is preserved when passing to a sub-σ-algebra, the existence of a conservative spectral representation of $\mathbf{X}$ implies that its minimal spectral representation is conservative, thus $\mathbf{X}^{\mathcal{D}}$ vanishes.
\end{remark}

Recall that a $G$-process $\mathbf{X}$ is {\it strongly-mixing} if for every pair of events $A,B$ in $\mathbf{X}$,
$$\lim_{g\in G}\mathbb{P}_{\mathbf{X}}\left(A\cap\lambda_{g}\left(B\right)\right)=\mathbb{P}_{\mathbf{X}}\left(A\right)\mathbb{P}_{\mathbf{X}}\left(B\right).$$
It is easy to see that a strongly-mixing $G$-process on a countable amenable group $G$ is weakly-mixing as in Definition \ref{Definition: Weak-mixing}.

Similarly to Theorem \ref{Theorem: weak-mixing null} about weak-mixing, the following theorem is convenient for checking whether a stationary $S\alpha S$ $G$-process is strongly-mixing. It was proved by Gross in \cite[Theorem~2.7]{gross} for $\mathbb{Z}$ and its proof remains valid for all countable groups.

\begin{theorem}[Gross]
\label{Theorem: strong-mixing characterization}
Let $G$ be a countable group and $\mathbf{X}$ a stationary non-Gaussian $S\alpha S$ $G$-process with (possibly non-minimal) spectral representation as in Definition \ref{Definition: Spectral Representation}. Then $\mathbf{X}$ is strongly-mixing if and only if, for every compact set $K\subset\left(0,\infty\right)$ and every $\epsilon>0$,
$$\lim_{g\in G}\mu\left(\left|f_{0}\right|^{\alpha}\in K,\left|\f u_g f_0\right|^{\alpha}>\epsilon\right)=0.$$
\end{theorem}

\subsection{Bernoulli shift constructions}

Let us introduce the (non-singular) Bernoulli shift model. Suppose that $G$ is a countable group and consider the space $\Sigma=\left\{0,1\right\}^G$ equipped with the usual product σ-algebra. For every $g\in G$ let $\rho_g$ be a distribution on $\left\{0,1\right\}$ such that $\rho_g\left(0\right)$ and $\rho_g\left(1\right)$ are non-zero, and consider the probability product measure $\rho=\bigotimes_{g\in G}\rho_g$ on $\Sigma$. The {\it shift} action of $G$ on $\Sigma$ is defined, for each $g\in G$ and $\omega=\left(\omega_{h}\right)_{h\in G}\in\Sigma$, by
$$\f s_{g}\left(\omega\right)=\left(\omega_{gh}\right)_{h\in G}.$$
When the shift is non-singular with respect to $\rho$, the $G$-system $G\overset{\f s}{\curvearrowright}\left(\Sigma,\rho\right)$ is called a (non-singular) {\it Bernoulli shift}. In order to check the non-singularity property for a Bernoulli shift one can use the fundamental {\it Kakutani Criterion} for mutual absolute continuity of probability product measures \cite{kakutani1948equivalence}, by which a pair of probability product measures $\rho=\bigotimes_{g\in G}\rho_{g}$ and $\varrho=\bigotimes_{g\in G}\varrho_{g}$ on $\Sigma$ are mutually absolutely continuous if and only if
\begin{equation}
\label{eq: Kakutani}
\sum_{g\in G}\sum_{i\in\left\{ 0,1\right\} }\left(\sqrt{\rho_{g}\left(i\right)}-\sqrt{\varrho_{g}\left(i\right)}\right)^{2}<\infty.
\end{equation}
(cf. \cite[\textsection 2]{vaeswahl2018} and \cite[\textsection 1.2]{bjorklundkosloff2018}).

The problem whether there is a conservative Bernoulli shift on $\mathbb{Z}$ that is null was answered positively few decades ago by Krengel \cite{krengel1970}, and this answer was strengthened by Hamachi \cite{hamachi1981bernoulli} who constructed such a Bernoulli shift without absolutely continuous invariant measure, also not a σ-finite one. See also \cite{kosloff2011, kosloff2014, danilenko2019}. Only in recent years, Vaes \& Wahl \cite{vaeswahl2018}; Björklund \& Kosloff \cite{bjorklundkosloff2018}; and Björklund, Kosloff \& Vaes \cite{bjorklundkosloffvaes2021} introduced examples of null conservative Bernoulli shifts on an arbitrary countable amenable group. Such an explicit construction can also be found in \cite[\textsection 6]{kosloffsoo2021}. Using those results we are able to show the following.

 \begin{proposition}
\label{Proposition: Existence}
Every countable amenable group $G$ admits, for every $\alpha\in\left(0,2\right)$, a non-zero stationary $S\alpha S$ $G$-process $\mathbf{X}$ that is strongly-mixing and is not a mixed moving average (namely, $\mathbf{X}^{\mathcal{D}}$ vanishes).
\end{proposition}

The main challenge is to show that every countable amenable group admits a {\it zero-type} conservative Bernoulli shift. Following \cite{kosloff2013}, we say that a Bernoulli shift $G\overset{\f s}{\curvearrowright}\left(\Sigma,\rho\right)$ is {\it zero-type} if
$$\lim_{g\in G}\int_{\Sigma}\sqrt{\frac{d\rho\circ\f s_{g}}{d\rho}}d\rho=0.$$ Here we use a construction of conservative Bernoulli shift due to Vaes \& Wahl \cite{vaeswahl2018} and apply the following lemma, whose proof in \cite[Lemma~4]{kosloff2013} is valid for all countable groups.

\begin{lemma}[Kosloff]
\label{Lemma: Kosloff Bernoulli shift is zero-type}
Let $G$ be a countable group and $G\overset{\f s}{\curvearrowright}\left(\Sigma,\rho\right)$ a Bernoulli shift. Suppose that $\rho=\bigotimes_{g\in G}\rho_{g}$ and the limit $\lim_{g\in G}\rho_g\left(i\right)=p\left(i\right)\neq 0$ exists for $i\in\left\{0,1\right\}$. Then either $\rho$ is mutually absolutely continuous with $\bigotimes_{g\in G}\left(p\left(0\right),p\left(1\right)\right)$, or that the Bernoulli shift $G\overset{\f s}{\curvearrowright}\left(\Sigma,\rho\right)$ is zero-type.
\end{lemma}

\begin{proof}[Proof of Proposition \ref{Proposition: Existence}]
Since $G$ is amenable, by \cite[Proposition~6.8]{vaeswahl2018} there exists a function $\Psi:G\to\left(0,1/3\right)$ such that, if we let
$$c_g=\left(\Psi\left(h\right)-\Psi\left(gh\right)\right)_{h\in G},\quad g\in G,$$
the following properties hold.
\begin{enumerate}
    \item $\lim_{g\in G}\Psi\left(g\right)=0$ and $\sum_{g\in G}\Psi\left(g\right)^{2}=\infty$;
    \item $c_g\in l^2\left(G\right)$ for all $g\in G$;
    \item $\lim_{g\in G}\no{c_{g}}_{l^{2}\left(G\right)}^{2}=\infty$ and $\sum_{g\in G}\exp\left(-45\no{c_{g}}_{l^{2}\left(G\right)}^{2}\right)=\infty.$
\end{enumerate}
Given such a function $\Psi$, define a probability product measure $\rho=\bigotimes_{g\in G}\rho_g$ on $\Sigma=\left\{0,1\right\}^{G}$ by
$$\rho_{g}\left(0\right)=1-\rho_{g}\left(1\right)={\textstyle \frac{1}{2}}+\Psi\left(g\right),\quad g\in G.$$
Note that $\rho_g\left(i\right)\geq 1/6$ for all $g\in G$ and $i\in\left\{0,1\right\}$. It follows that, for every $g\in G$,
\begin{align*}
    {\textstyle \sum_{h\in G}\sum_{i\in\left\{ 0,1\right\} }}\left(\sqrt{\rho_{h}\left(i\right)}-\sqrt{\rho_{gh}\left(i\right)}\right)^{2}
    & \leq{\textstyle \frac{3}{2} \sum_{h\in G}\sum_{i\in\left\{ 0,1\right\} }}\left(\rho_{h}\left(i\right)-\rho_{gh}\left(i\right)\right)^{2} \\
    & ={\textstyle \frac{3}{2}}\no{c_{g}}_{l^{2}\left(G\right)}^{2}<\infty,
\end{align*}
so from the Kakutani Criterion \eqref{eq: Kakutani} it follows that $\rho$ and $\rho\circ\f s_g$ are mutually absolutely continuous. Thus, the resulting Bernoulli shift is non-singular. Since $\rho_g\left(i\right)\geq 1/6$ for all $g\in G$ and $i\in\left\{0,1\right\}$, it follows from \cite[Proposition~4.1(1)]{vaeswahl2018} that since $\sum_{g\in G}\exp\left(-\kappa\no{c_{g}}_{l^{2}\left(G\right)}^{2}\right)=\infty$ with $\kappa=45$, the resulting Bernoulli shift is conservative.

We claim that this Bernoulli shift is zero-type. From the properties of $\Psi$ we see that $\lim_{g\in G}\rho_g\left(i\right)=1/2$ for $i\in\left\{0,1\right\}$ and that
\begin{align*}
{\textstyle \sum_{g\in G}\sum_{i\in\left\{ 0,1\right\} }}\left(\sqrt{\rho_{g}\left(i\right)}-\sqrt{1/2}\right)^{2}
& \geq {\textstyle \frac{1}{4} \sum_{g\in G}\sum_{i\in\left\{ 0,1\right\} }}\left(\rho_{g}\left(i\right)-1/2\right)^{2}\\
& ={\textstyle \frac{1}{2}\sum_{g\in G}}\Psi\left(g\right)^{2}=\infty.
\end{align*}
It follows from the Kakutani Criterion \eqref{eq: Kakutani} that $\rho$ is not mutually absolutely continuous with $\bigotimes_{g\in G}\left(1/2,1/2\right)$. Then from Lemma \ref{Lemma: Kosloff Bernoulli shift is zero-type} the resulting Bernoulli shift is zero-type.

Finally, given some $\alpha\in\left(0,2\right)$, let $M_{\alpha}$ be an $S\alpha S$ random measure controlled by $\rho$ and define
$$\mathbf{X}=\left(\int_{\Omega}\sqrt[\alpha]{\frac{d\rho\circ\f s_{g}}{d\rho}} dM_{\alpha}\right)_{g\in G}.$$
This is a stationary $S\alpha S$ $G$-process and the Bernoulli shift $G\overset{\f s}{\curvearrowright}\left(\Sigma,\rho\right)$ constructed above is the $G$-system underlying its Rosiński minimal spectral representation. Since this Bernoulli shift is conservative, and in view of Remark \ref{Remark: conservative part}, we see that $\mathbf{X}^{\mathcal{D}}$ vanishes. Since the underlying Bernoulli shift is zero-type, for every $\epsilon>0$,
$$\rho\left(\frac{d\rho\circ\f s_{g}}{d\rho}>\epsilon\right)\leq\epsilon^{-1/2}\int_{\Sigma}\sqrt{\frac{d\rho\circ\f s_{g}}{d\rho}}d\rho\xrightarrow[g\in G]{}0.$$
By a simple application of Theorem \ref{Theorem: strong-mixing characterization} with $f_0=1$ we see that $\mathbf{X}$ is strongly-mixing.
\end{proof}

\subsection{Infinite measure preserving constructions}

Recall that a $G$-system is {\it ergodic} if every invariant integrable function on the underlying measure space is almost surely constant. Let $\ensuremath{G}\overset{\f a}{\curvearrowright}\left(\Omega,\mu\right)$ be an infinite (σ-finite, as always) measure preserving ergodic $G$-system. Note that such a $G$-system is always null. We say that $\ensuremath{G}\overset{\f a}{\curvearrowright}\left(\Omega,\mu\right)$ is {\it rigid} if there is an infinite sequence $\left(g_n\right)_{n=1}^{\infty}$ in $G$ such that, for every Borel set $E$ in $\Omega$ with $\mu\left(E\right)<\infty$,
\begin{equation}
    \label{eq: rigid}
    \lim_{n\to\infty}\mu\left(E\cap\f a_{g_n}\left(E\right)\right)=\mu\left(E\right).
\end{equation}

Danilenko used the method of $\left(C,F\right)$-construction in order to present various infinite measure preserving systems of groups that enjoy special recurrence properties. Here we exploit the constructions in \cite{danilenko2016, danilenko2017} and, combined with Theorem \ref{Theorem: weak-mixing null} for weak-mixing and Theorem \ref{Theorem: strong-mixing characterization} for strong-mixing, we obtain the following result. Recall that the Heisenberg group $H_3\left(\mathbb{Z}\right)$ is the group of $3\times 3$ upper triangular matrices with integer entries and $1$ on the main diagonal. This is a nilpotent group of order $2$ and in particular it is amenable.

\begin{proposition}
Let $\Gamma$ be either a countable Abelian group with an element of infinite order or the Heisenberg group. Then $\Gamma$ admits, for every $\alpha\in\left(0,2\right)$, a stationary $S\alpha S$ $\Gamma$-process that is weakly-mixing but not strongly-mixing.
\end{proposition}

\begin{proof}
From \cite[Theorem~0.1]{danilenko2016} (when $\Gamma$ is a countable Abelian group with an element of infinite order) and  \cite[Theorem~7.3]{danilenko2017} (when $\Gamma=H_3\left(\mathbb{Z}\right)$ is the Heisenberg group), there exists an essentially free, infinite measure preserving $\Gamma$-system $\ensuremath{\Gamma}\overset{\f a}{\curvearrowright}\left(\Omega,\mu\right)$ that is ergodic and rigid.

Given $\alpha\in\left(0,2\right)$, let $M_{\alpha}$ be an $S\alpha S$ random measure controlled by $\mu$, choose an arbitrary bounded real-valued $f_0\in L^{\alpha}\left(\mu\right)$ supported on all of $\Omega$ and define
$$\mathbf{X}=\left(\int_{\Omega}f_{0}\circ\f a_{\gamma}dM_{\alpha}\right)_{\gamma\in\Gamma}.$$
This is a stationary $S\alpha S$ $\Gamma$-process and $\Gamma\overset{\f a}{\curvearrowright}\left(\Omega,\mu\right)$ is the measure preserving $\Gamma$-system underlying its Rosiński minimal spectral representation.

Since $\ensuremath{\Gamma}\overset{\f a}{\curvearrowright}\left(\Omega,\mu\right)$ is null, as all infinite measure preserving ergodic systems, it follows from Theorem \ref{Theorem: weak-mixing null} that $\mathbf{X}$ is weakly-mixing. Let us show that $\mathbf{X}$ is not strongly-mixing. As $\left|f_0\right|^{\alpha}$ is bounded and integrable, there is a sufficiently small $\epsilon>0$ such that $\left|f_0\right|^{\alpha}<\epsilon^{-1}$ almost surely and the set $E:=\left\{\left|f_0\right|^{\alpha}>\epsilon\right\}$ satisfies $0<\mu\left(E\right)<\infty$. Let $\left(\gamma_n\right)_{n=1}^{\infty}$ be an infinite sequence in $\Gamma$ that is satisfying the rigidity property of $\Gamma\overset{\f a}{\curvearrowright}\left(\Omega,\mu\right)$ as in \eqref{eq: rigid}. Then
\begin{align*}
    \mu\left(\left|f_{0}\right|^{\alpha}\in\left[\epsilon,\epsilon^{-1}\right],\left|f_{0}\circ\f a_{\gamma_{n}^{-1}}\right|^{\alpha}>\epsilon\right)
    & =\mu\left(\left|f_{0}\right|^{\alpha}\geq\epsilon,\left|f_{0}\circ\f a_{\gamma_{n}^{-1}}\right|^{\alpha}>\epsilon\right) \\
    & \geq\mu\left(E\cap\f a_{\gamma_{n}}\left(E\right)\right)\xrightarrow[n\to\infty]{}\mu\left(E\right)>0.
\end{align*}
Thus, from Theorem \ref{Theorem: strong-mixing characterization} it follows that $\mathbf{X}$ is not strongly-mixing.
\end{proof}

\begin{appendix}

\section*{The Dye-Douglass criterion}
\label{Appendix: Dye-Douglass Criterion}

The proof Douglass gave in \cite[Theorem~4.1]{douglass} to Theorem \ref{Theorem: Dye-Douglass Criterion} made an essential use of a non-trivial statement by Dye in \cite[\textsection 1]{dye} that we could not locate a proof of it in the literature. Here we introduce the Dye-Douglass Criterion with a proof of Dye's statement.

We start with some notations. Let $G$ be a countable amenable group. Denote by $\mathcal{P}\left(G\right)$ the set of finitely supported probability measures on $G$, viewed as a convex subset of $l^{1}\left(G\right)\subset l^{\infty}\left(G\right)$, so we can freely write $\lambda_{g}p$ and $\rho_{g}p$ for $p\in\mathcal{P}\left(G\right)$. For $F\Subset G$ define $p_{F}\in\mathcal{P}\left(G\right)$ by
$$p_{F}\left(g\right)={\textstyle \frac{1}{\left|F\right|}}1_{F}\left(g\right),\quad g\in G.$$
For $k\in G$ denote $G_{k}=\left\{ \left(g,h\right)\in G\times G:gh=k\right\}$. For $p,q\in\mathcal{P}\left(G\right)$ define their convolution $p\ast q\in\mathcal{P}\left(G\right)$ by
$$p\ast q\left(k\right)=\sum_{\left(g,h\right)\in G_k}p\left(g\right)q\left(h\right),\quad g\in G.$$
For $\psi\in l^{\infty}\left(G\right)$ and $p\in\mathcal{P}\left(G\right)$ write
$$E_{p}\left(\psi\right)=\sum_{g\in G}p\left(g\right)\varphi\left(g\right).$$
Note the formula
$$E_{p\ast q}\left(\psi\right)=\sum_{k\in G}\sum_{\left(g,h\right)\in G_k}p\left(g\right)q\left(h\right)\psi\left(gh\right)=\sum_{g,h\in G}p\left(g\right)q\left(h\right)\psi\left(gh\right).$$

\begin{definition}[Dye \cite{dye}]
For a real-valued function $\psi\in l^{\infty}\left(G\right)$ let
$$\lambda^{-}\left(\psi\right)=\sup_{p\in\mathcal{P}\left(G\right)}\inf_{q\in\mathcal{P}\left(G\right)}E_{q\ast p}\left(\psi\right)=\sup_{p\in\mathcal{P}\left(G\right)}\inf_{h\in G}E_{p}\left(\lambda_{h}\psi\right);$$
$$\lambda^{+}\left(\psi\right)=\inf_{p\in\mathcal{P}\left(G\right)}\sup_{q\in\mathcal{P}\left(G\right)}E_{q\ast p}\left(\psi\right)=\inf_{p\in\mathcal{P}\left(G\right)}\sup_{h\in G}E_{p}\left(\lambda_{h}\psi\right);$$
$$\rho^{-}\left(\psi\right)=\sup_{p\in\mathcal{P}\left(G\right)}\inf_{q\in\mathcal{P}\left(G\right)}E_{p\ast q}\left(\psi\right)=\sup_{p\in\mathcal{P}\left(G\right)}\inf_{h\in G}E_{p}\left(\rho_{h}\psi\right);$$
$$\rho^{+}\left(\psi\right)=\inf_{p\in\mathcal{P}\left(G\right)}\sup_{q\in\mathcal{P}\left(G\right)}E_{p\ast q}\left(\psi\right)=\inf_{p\in\mathcal{P}\left(G\right)}\sup_{h\in G}E_{p}\left(\rho_{h}\psi\right).$$
\end{definition}

Note that always $\lambda^{-}\left(\psi\right)\leq\lambda^{+}\left(\psi\right)$ and $\rho^{-}\left(\psi\right)\leq\rho^{+}\left(\psi\right)$. Each of the second equalities appeared in the above definition are due to an observation of Douglass. In order to see why this is true let us explain, for instance, why
$$\sup_{q\in\mathcal{P}\left(G\right)}E_{q\ast p}\left(\psi\right)=\sup_{h\in G}E_{p}\left(\lambda_{h}\psi\right)$$
for every $p\in\mathcal{P}\left(G\right)$. Indeed, one inequality follows by replacing the supremum over $\mathcal{P}\left(G\right)$ by the supremum over the Dirac measures on $G$. For the converse inequality note that for every $q\in\mathcal{P}\left(G\right)$ we have
$$E_{q\ast p}\left(\psi\right)={\textstyle \sum_{h\in G}}q\left(h\right)E_{p}\left(\lambda_{h}\psi\right)\leq\sup_{h\in G}E_{p}\left(\lambda_{h}\psi\right).$$

Recalling Definition \ref{Definition: Almost Convergence} of almost convergence and universal mean and Definition \ref{Definition: Folner Convergence} of Følner convergence and Følner mean, we formulate the Dye-Douglass Criterion as follows.

\begin{theorem}[The Dye-Douglass Criterion]
Let $G$ be a countable amenable group. For a real-valued function $\psi\in l^{\infty}\left(G\right)$ the following are equivalent.
\begin{enumerate}
\item $\lambda^{-}\left(\psi\right)=\lambda^{+}\left(\psi\right)$;
\item $\rho^{-}\left(\psi\right)=\rho^{+}\left(\psi\right)$;
\item $\psi$ is Følner convergent;
\item $\psi$ is almost convergent.
\end{enumerate}
In case that $\psi$ satisfies these properties, its universal mean, its Følner mean, the common value of $\lambda^{-}\left(\psi\right)$ and $\lambda^{+}\left(\psi\right)$ as well as the common value of $\rho^{-}\left(\psi\right)$ and $\rho^{+}\left(\psi\right)$ all coincide.
\end{theorem}

\begin{proof}
$\left(1\right)\vee\left(2\right)\implies\left(3\right)$. This was proved by Douglass. We shortly mention how $\left(1\right)\implies\left(3\right)$ and the implication $\left(2\right)\implies\left(3\right)$ is similar. By \cite[Lemma~3.1]{douglass}, if $\left(F_{N}\right)_{N=1}^{\infty}$ is a left-Følner sequence for $G$, then for every $p\in\mathcal{P}\left(G\right)$,
$$\lim_{N\to\infty}\no{p\ast p_{F_{N}}-p_{F_{N}}}_{l^{1}\left(G\right)}=0.$$
When $\left(F_{N}\right)_{N=1}^{\infty}$ is a two-sided Følner sequence then by \cite[Lemma~4.1]{douglass}, for every real-valued $\psi\in l^{\infty}\left(G\right),$
$$\lambda^{-}\left(\psi\right)\leq\liminf_{N\to\infty}\inf_{h\in G}E_{p_{F_{N}}}\left(\lambda_{h}\psi\right)\leq\limsup_{N\to\infty}\sup_{h\in G}E_{p_{F_{N}}}\left(\lambda_{h}\psi\right)\leq\lambda^{+}\left(\psi\right).$$
It is now clear that $\lambda^{-}\left(\psi\right)=\lambda^{+}\left(\psi\right)$ implies that $\psi$ is Følner convergent.

$\left(3\right)\implies\left(4\right)$. Let $\mathbf{m}$ be a left-invariant mean on $G$ and fix some two-sided Følner sequence $\left(F_{N}\right)_{N=1}^{\infty}$ for $G$. For every $N$ let $\psi_{N}:G\to\mathbb{R}$ be the function
$$\psi_{N}\left(h\right)=E_{p_{F_{N}}}\left(\rho_{h^{-1}}\psi\right)={\textstyle \frac{1}{\left|F_{N}\right|}\sum_{g\in F_{N}}}\rho_{h^{-1}}\psi\left(g\right)={\textstyle \frac{1}{\left|F_{N}\right|}\sum_{g\in F_{N}}}\lambda_{g}\psi\left(h\right).$$
Thus, $\psi_{N}$ is a convex combination of left-translations of $\psi$. As $\mathbf{m}$ is left-invariant, $\mathbf{m}\left(\psi_{N}\right)=\mathbf{m}\left(\psi\right)$ for every $N$. Since $\mathbf{m}$ has norm $1$ as a linear functional on $l^{\infty}\left(G\right)$ we obtain that
$$\left|\mathbf{m}\left(\psi\right)-E_{G}^{\text{Følner}}\left(\psi\right)\right|=\left|\mathbf{m}\left(\psi_{N}-E_{G}^{\text{Følner}}\left(\psi\right)\right)\right|\leq\no{\psi_{N}-E_{G}^{\text{Følner}}\left(\psi\right)}_{l^{\infty}\left(G\right)}.$$
However, the assumption that $\psi$ is Følner convergent means that $\left(\psi_{N}\right)_{N=1}^{\infty}$ converges to $E_{G}^{\text{Følner}}\left(\psi\right)$ uniformly, hence $\mathbf{m}\left(\psi\right)=E_{G}^{\text{Følner}}\left(\psi\right)$. A similar argument shows that $\mathbf{m}\left(\psi\right)=E_{G}^{\text{Følner}}\left(\psi\right)$ also for every right-invariant mean $\mathbf{m}$.

$\left(4\right)\implies\left(1\wedge 2\right)$. We show that $\left(4\right)\implies\left(1\right)$ and the implication $\left(4\right)\implies\left(2\right)$ is similar. Suppose to the contrary that $\lambda^{-}\left(\psi\right)<\lambda^{+}\left(\psi\right)$. Thus, for every $p\in\mathcal{P}\left(G\right)$ we have
\begin{equation}
    \label{eq: L-<L+}
    \inf_{h\in G}E_{p}\left(\lambda_{h}\psi\right)\leq\lambda^{-}\left(\psi\right)<\lambda^{+}\left(\psi\right)\leq\sup_{h\in G}E_{p}\left(\lambda_{h}\psi\right).
\end{equation}
Put $\epsilon=\left(\lambda^{+}\left(\psi\right)-\lambda^{-}\left(\psi\right)\right)/3$. Let $\left(F_{N}\right)_{N=1}^{\infty}$ be an arbitrary two-sided Følner sequence for $G$. For every $N$, considering the element $p_{F_{N}}\in\mathcal{P}\left(G\right)$, in view of \eqref{eq: L-<L+} we see that there can be found group elements $h_{N}^{-}$ and $h_{N}^{+}$ such that
\begin{equation}
    \label{eq:L-L+}
    E_{p_{F_{N}}}\left(\lambda_{h_{N}^{-}}\psi\right)<\lambda^{-}\left(\psi\right)+\epsilon<\lambda^{+}\left(\psi\right)-\epsilon<E_{p_{F_{N}}}\left(\lambda_{h_{N}^{+}}\psi\right).
\end{equation}
Let $\mathcal{L}$ be some Banach limit on $l^{\infty}\left(\mathbb{N}\right)$, the Banach space of bounded sequences indexed by the positive integers. That is to say, $\mathcal{L}$ is a shift-invariant positive linear functional on $l^{\infty}\left(\mathbb{N}\right)$ that, on the subspace of convergent sequences, coincides with the usual limit functional. Such a linear functional satisfies, for every real-valued sequence $\left(x_{N}\right)_{N=1}^{\infty}\in l^{\infty}\left(\mathbb{N}\right)$,
$$\liminf_{N\to\infty}x_{N}\leq\mathcal{L}\left(\left(x_{N}\right)_{N=1}^{\infty}\right)\leq\limsup_{N\to\infty}x_{N}.$$
For the existence and the properties of Banach limits see for instance \cite[\textsection 3.4]{krengel1985}. Having $\mathcal{L}$ in hands, for $s\in\left\{ +,-\right\} $ define a mean $\mathbf{m}^{s}$ on $G$ by
$$\mathbf{m}^{s}\left(b\right)=\mathcal{L}\left(\left(E_{p_{F_{N}}}\left(\lambda_{h_{N}^{s}}b\right)\right)_{N=1}^{\infty}\right),\quad b\in l^{\infty}\left(G\right).$$
From the properties of $\mathcal{L}$ it is easy to see that $\mathbf{m}^{s}$ is a mean. We claim that $\mathbf{m}^{s}$ is right-invariant. Indeed, given $b\in l^{\infty}\left(G\right)$ and $h\in G$, for every $N$ we have
\begin{align*}
    E_{p_{F_{N}}}\left(\lambda_{h_{N}^{s}}\rho_{h}b\right)	
    & = {\textstyle \frac{1}{\left|F_{N}\right|}\sum_{g\in F_{N}h}}\lambda_{h_{N}^{s}}b\left(g\right) \\
    & = E_{p_{F_{N}}}\left(\lambda_{h_{N}^{s}}b\right)+{\textstyle \frac{1}{\left|F_{N}\right|}\sum_{g\in F_{N}h\backslash F_{N}}}\lambda_{h_{N}^{s}}b\left(g\right)-{\textstyle \frac{1}{\left|F_{N}\right|}\sum_{g\in F_{N}\backslash F_{N}h}}\lambda_{h_{N}^{s}}b\left(g\right).
\end{align*}
Since $\left(F_{N}\right)_{N=1}^{\infty}$ is two-sided and $b\in l^{\infty}\left(G\right)$, from the Følner property the second and third terms vanish as $N\to\infty$. As $\mathcal{L}$ assigns zero to sequences converging to zero, we conclude that $\mathbf{m}^{s}\left(\rho_{h}b\right)=\mathbf{m}^{s}\left(b\right)$. Finally, in view of \eqref{eq:L-L+} and the properties of $\mathcal{L}$ we deduce that
\begin{equation}
    \mathbf{m}^{-}\left(\psi\right)\leq\lambda^{-}\left(\psi\right)+\epsilon<\lambda^{+}\left(\psi\right)-\epsilon\leq\mathbf{m}^{+}\left(\psi\right)\nonumber
\end{equation}
showing that $\psi$ is not almost convergent.
\end{proof}

\section*{A lemma on almost-sure convergence}
\label{Appendix: Lemma on Almost-Sure}

The following auxiliary lemma can be deduced from the considerations appear in the proof of \cite[Theorem~1.4.4]{aaronson}. We introduce the proof here for completeness.

\begin{lemma}
Let $\left(\Omega,\mu\right)$ be a standard probability space and $\left(f_{k}\right)_{k=1}^{\infty}\subset L^{1}\left(\mu\right)$ a bounded sequence of real-valued non-negative functions that is Cesàro convergent to $0$ almost surely, namely
$$\frac{1}{K}\sum_{k=1}^{K}f_{k}\xrightarrow[K\to\infty]{}0\text{ almost surely}.$$
Then $\left(f_{k}\right)_{k=1}^{\infty}$ has a subsequence convergent to $0$ almost surely.
\end{lemma}

\begin{proof}
For every positive integer $m$ define $b_{m}\in\left\{ f_{k}:k\in\mathbb{N}\right\}$ as follows. Using Egorov's Theorem choose a Borel set $E_{m}$ in $\Omega$ with $\mu\left(\Omega\backslash E_{m}\right)<2^{-m}$ such that on $E_{m}$, the almost sure Cesàro convergence of $\left(f_{k}\right)_{k=1}^{\infty}$ to $0$ is uniformly in $k$. In particular, $\left(f_{k}\right)_{k=1}^{\infty}$ is Cesàro convergent to $0$ in $L^{1}\left(\mu\right)$ on $E_{m}$, so there is a sequence $\left(K\left(m,r\right)\right)_{r=1}^{\infty}$ of integers such that
$${\textstyle \frac{1}{K\left(m,r\right)}\sum_{k=1}^{K\left(m,r\right)}}\int_{E_{m}}f_{k}d\mu<2^{-r},\quad r=1,2,\dotsc.$$
Then we can extract a sequence $\left(k\left(m,r\right)\right)_{r=1}^{\infty}$ of integers, $1\leq k\left(m,r\right)\leq K\left(m,r\right)$, such that
$$\int_{E_{m}}f_{k\left(m,r\right)}d\mu<2^{-r},\quad r=1,2,\dotsc.$$
Then define
$$b_{m}:=f_{k\left(m,m\right)}.$$
Doing this for every $m$, after a simple adjustment we can make $\left(b_{m}\right)_{m=1}^{\infty}$ a subsequence of $\left(f_{k}\right)_{k=1}^{\infty}$. We show that it is convergent to $0$ almost surely. Let $\epsilon>0$ be arbitrary. Using Markov's inequality, for every $m$ we have
\begin{align*}
    \mu\left(b_{m}\geq\epsilon\right)
    & \leq \mu\left(\Omega\backslash E_{m}\right)+\mu\left(E_{m}\cap\left\{ b_{m}\geq\epsilon\right\} \right) \\
    & < 2^{-m}+\epsilon^{-1}\int_{E_{m}}b_{m}d\mu<\left(1+\epsilon^{-1}\right)2^{-m}.
\end{align*}
From the Borel-Cantelli lemma we conclude that $b_{m}\xrightarrow[m\to\infty]{}0$ almost surely.
\end{proof}

\end{appendix}

\begin{acks}[Acknowledgments]
I am in debt to my advisor, Zemer Kosloff, for introducing me to P. Roy's work \cite{roy2020}, for many helpful discussions and for his dedication and constant support throughout this research. I thank Amichai Lampert and Hagai Lavner for a few helpful discussions about the {\it Følner convergence} property. I also thank the referee for their time and their useful comments. A special thanks goes to Mahan Mj, Parthanil Roy and Sourav Sarkar who, after this work was first published, kindly let me know that they are working on similar problems. The research was partially supported by ISF grant No. 1180/22.
\end{acks}

\bibliographystyle{imsart-number}
\bibliography{References.bib}

\end{document}